%% file: Bubbles_v1.tex
\documentclass[12pt]{amsart}
\usepackage[utf8]{inputenc}
\usepackage[english]{babel}

\usepackage[left=25mm, right=25mm, top=25mm, bottom=25mm]{geometry}
\usepackage[shortlabels]{enumitem}
\usepackage{polynom, subcaption, hyperref, color, comment}
\hypersetup{colorlinks, citecolor=red, filecolor=black, linkcolor=blue, urlcolor=blue}
\usepackage{tcolorbox}
\usepackage{appendix}
\usepackage{graphicx, float, tikz-cd, tikz, caption}

\usepackage{csquotes}

\makeatletter
\def\Ddots{\mathinner{\mkern1mu\raise\p@
\vbox{\kern7\p@\hbox{.}}\mkern2mu
\raise4\p@\hbox{.}\mkern2mu\raise7\p@\hbox{.}\mkern1mu}}
\makeatother

\usepackage{amsmath, amsfonts, amssymb, amsthm, mathtools, mathrsfs, commath}

\newtheorem {theorem}{Theorem}[section]
\newtheorem {lemma}[theorem]{Lemma}
\newtheorem {proposition}[theorem]{Proposition}
\newtheorem {conjecture}[theorem]{Conjecture}
\newtheorem*{conjecture*}{Conjecture}

\theoremstyle{definition}
\newtheorem {definition}[theorem]{Definition}
\newtheorem {remark}[theorem]{Remark}

\newtheorem{thm}{Theorem}

\begin{document}

\title{A counterexample to the singular Weinstein conjecture}
\author{Josep Fontana-McNally}
\thanks{Josep Fontana-McNally was supported by an INIREC grant Introduction to research financed under the  project “Computational, dynamical and geometrical complexity in fluid dynamics”, Ayudas Fundación BBVA a Proyectos de Investigación Científica 2021. Josep Fontana, Eva Miranda and Cédric Oms are partially supported by the Spanish State Research Agency grant PID2019-103849GB-I00 of AEI / 10.13039/501100011033 and by the AGAUR project 2021 SGR 00603.}
\address{{Laboratory of Geometry and Dynamical Systems, Department of Mathematics}, Universitat Polit\`{e}cnica de Catalunya and University of Oxford}
\email{josep.fontana@estudiantat.upc.edu}

\author{ Eva Miranda}
\thanks{ Eva Miranda is supported by the Catalan Institution for Research and Advanced Studies via an ICREA Academia Prize 2021 and by the Alexander Von Humboldt Foundation via a Friedrich Wilhelm Bessel Research Award. Eva Miranda is also supported by the Spanish State
Research Agency, through the Severo Ochoa and Mar\'{\i}a de Maeztu Program for Centers and Units
of Excellence in R\&D (project CEX2020-001084-M).  Eva Miranda and Daniel Peralta-Salas  acknowledge partial support from the grant
“Computational, dynamical and geometrical complexity in fluid dynamics”, Ayudas Fundación BBVA a Proyectos de Investigación Científica 2021.  }
\address{{Laboratory of Geometry and Dynamical Systems \& IMTech, Department of Mathematics}, Universitat Polit\`{e}cnica de Catalunya and CRM, Barcelona, Spain \\ Centre de Recerca Matemàtica-CRM}
\email{eva.miranda@upc.edu}

\author{Cédric Oms}
\thanks{Cédric Oms acknowledges financial support from the Juan de la Cierva post-doctoral grant (grant number FCJ2021-046811-I)}
\address{BCAM Bilbao, Mazarredo, 14. 48009 Bilbao Basque Country - Spain}
\email{coms@bcamath.org}

\author{Daniel Peralta-Salas}
\address{Instituto de Ciencias Matemáticas (ICMAT), Consejo Superior de Investigaciones Científicas, Madrid}
\email{dperalta@icmat.es}
\thanks{Daniel Peralta-Salas is supported by the grants CEX2019-000904-S, RED2022-134301-T and PID2022-136795NB-I00 funded by
MCIN/AEI/10.13039/501100011033.}
\maketitle

\begin{abstract}
In this article, we study the dynamical properties of Reeb vector fields on $b$-contact manifolds. We show that in dimension $3$, the number of so-called singular periodic orbits can be prescribed. These constructions illuminate some key properties of escape orbits and singular periodic orbits, which play a central role in formulating singular counterparts to the Weinstein conjecture and the Hamiltonian Seifert conjecture. In fact, we prove that the above-mentioned constructions lead to counterexamples of these conjectures as stated in~\cite{MO21}. Our construction shows that there are $b$-contact manifolds with no singular periodic orbit and no regular periodic orbit away from $Z$. We do not know whether there are constructions with no generalized escape orbits whose $\alpha$ \textit{and} $\omega$-limits both lie on $Z$ (a generalized singular periodic orbit). This is the content of the \emph{generalized Weinstein conjecture}.
\end{abstract}


\section{Introduction}

Compactifying the phase space of dynamical systems by adding manifolds at infinity or collision manifolds yields singularities in the geometric forms that define the flow of the systems. These singularities are in general of $E$-type \cite{MS21}. In various instances within celestial mechanics, they emerge as a well-understood sub-type, commonly referred to as ``type $b$''-singularities, following regularization transformations. 

The geometric framework for $E$-singularities developed in \cite{MS21} is powerful enough to include the possibility of defining Hamiltonian and Reeb vector fields associated with singular symplectic and contact forms respectively. These vector fields are smooth vector fields that extend the dynamics on a space to the singular manifold, and in some respects behave similarly to their counterparts defined using smooth symplectic and contact forms. This leads to the question: do central theorems about periodic orbits of Hamiltonian and Reeb vector fields also hold when these are defined by singular structures? 
Exploring this question is pivotal for comprehending dynamics on non-compact manifolds, as they can often be compactified into  $E$-manifolds.

In \cite{MO21}, Miranda and Oms examined this question for the special case of $b$-structures and formulated the $b$-counterpart of the Weinstein and the Hamiltonian Seifert conjectures, which we shall call the singular Weinstein and singular Hamiltonian Seifert conjectures throughout this article. In Sections \ref{sec:counterexample} and \ref{sec:bSeifert} we tackle these conjectures, providing counterexamples for them using techniques developed in Section \ref{sec:tools}.

The question about dynamical invariant sets on compact $b$-contact manifolds is two-fold. On the one hand, there always exist infinitely many periodic orbits on the critical set when the dimension of the ambient manifold is $3$ \cite{MO21}; this comes from the fact that the dynamics on the critical set is described by the Hamiltonian dynamics of a certain function, called the \emph{exceptional Hamiltonian}. On the other hand, the article \cite{MO21} provides examples of closed $b$-contact manifolds whose associated Reeb vector field does not have any periodic orbits outside of the singular set. Consequently, the question of existence of periodic orbits \emph{outside the critical set}  needs to be reformulated.

To pose and answer such questions of dynamics on singular manifolds, it is convenient to make the following distinctions between ``singular orbits", that is, integral curves with limit sets on the critical set (see also Definition \ref{def:spo}). In loose order of generality, a \textit{generalized escape orbit} is an orbit whose $\alpha$- \textit{or} $\omega$-limit set has a non-empty intersection with the critical set. An \textit{escape orbit} is an orbit whose $\alpha$- \textit{or} $\omega$-limit set is \textit{a point} on the critical set. A \textit{generalized singular periodic orbit}  is an orbit whose $\alpha$- \textit{and} $\omega$-limit sets lie on the singular set (see Definition \ref{def:gspo}). Finally, a \textit{singular periodic orbit} is an orbit whose $\alpha$- \textit{and} $\omega$-limit sets are \textit{points} on the critical set. The following diagram illustrates the relationships between these types of singular orbits (see also Figure \ref{fig:gardenoforbits} for pictures of each type of orbit). We note that generalized singular periodic orbits and singular periodic orbits are called so, because if we quotient the critical set to a point in the whole manifold, these orbits, together with the singular point, become closed invariant 1-dimensional submanifolds. In the language of dynamical systems, in the quotient space, those orbits correspond to homoclinic orbits. 

\begin{figure}[H]\label{fig:dependenciessingularorbits}
    \centering
    \input{Fig_DependenciesSingularOrbits}
    \caption{Dependencies between the different types of singular orbits. None of the converse implications hold.}
    \label{fig:enter-label}
\end{figure}
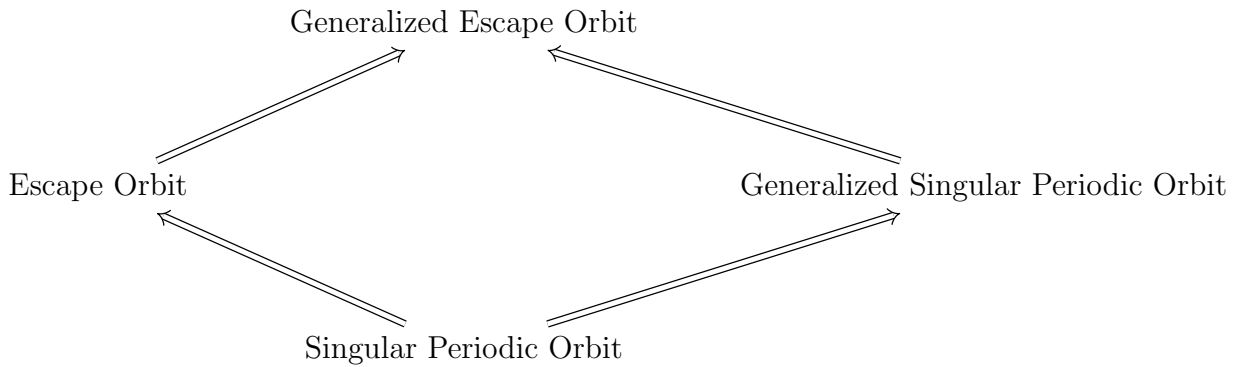

In \cite{MO21}, it was conjectured that the right dynamical invariants to consider in $b$-contact manifolds are \textit{singular periodic orbits}. These are orbits contained outside of the critical set whose $\alpha$- and $\omega$-limit sets are exactly stationary points of the Reeb $b$-vector field on the critical set (see Figure \ref{fig:spodef} below). The singular Weinstein conjecture states that there exists at least a periodic orbit or a singular periodic orbit outside of the critical set. This conjecture was supported by the examples not satisfying the regular Weinstein conjecture. Further progress was obtained in \cite{MOP22,FMOP23}, where the authors analyzed the dynamical properties of the Reeb $b$-vector field around the singular points of the field on the critical set.


Loosely speaking, it was shown in \cite{MOP22} that generically, there exist escape orbits around the stationary points of the Reeb $b$-vector field on the critical set. Later, a more thorough analysis in \cite{FMOP23} showed that the number of such orbits is infinite if the exceptional Hamiltonian has critical points of Morse index $1$. These results point towards a positive answer to the singular Weinstein conjecture.

However, contrary to initial expectations, in this article, we show that the singular Weinstein conjecture does not hold. 

\begin{thm}[Theorem \ref{thm:counterexample}]\label{thm:introcounterexample}
    There exists a $b$-contact form on $(\mathbb{S}^3,\mathbb{S}^2)$ such that the flow of the associated Reeb $b$-vector field does not have any singular periodic orbits, nor periodic orbits outside of the critical set.
\end{thm}

Note that this counterexample does not exclude generalized singular periodic orbits. The counterexample is based on analyzing the flow of the Reeb $b$-vector field induced on the unit sphere in the standard $b$-symplectic $\mathbb{R}^{4}$. In analogy to the smooth case, we call this flow the $b$-Hopf flow on $\mathbb S^3$. As already observed in \cite{MO21}, this flow has $2$ singular periodic orbits, and no periodic orbits outside the critical set. However, in Section \ref{sec:tools} we show that we can alter the $b$-contact form in small coordinate neighborhoods through which these singular periodic orbits pass, in such a way that the singular periodic orbits are broken off. Thus, we obtain a $b$-contact manifold satisfying the desired conditions. Indeed, in Section \ref{sec:tools} we show that we have a lot of \emph{flexibility} to customize the number of singular periodic orbits as we prove in the following theorem.

\begin{thm}[Theorem \ref{thm:nspo}]\label{thmm:nspo}
Let $(M,\xi)$ be a coorientable $3$-dimensional contact manifold. Then for any integer $0\leq k\leq N$, there exists a $b$-contact form on $M$ whose critical set $Z$ consists of $N$ components diffeomorphic to $\mathbb S^2$ and the number of singular periodic orbits is exactly $k$. The associated $b$-contact structure coincides with $\xi$ outside a neighborhood of the balls enclosed by $Z$. Furthermore, there is an infinite number of generalized escape orbits converging to each component of $Z$ (i.e., orbits whose $\alpha$- or $\omega$-limit sets are on $Z$ but are not singular points).
\end{thm}



In Section \ref{sec:bSeifert}, we then show that these results can be used to construct a $1$-parametric version of Theorem \ref{thm:introcounterexample} in the standard $b$-symplectic $\mathbb{R}^{4}$.

\begin{thm}[Theorem \ref{thm:R4counterexampleSeifert}]
    There exists a smooth Hamiltonian $H$ in $(\mathbb{R}^4,{^b}\omega_0)$ such that for $\epsilon>0$ small enough, $\Sigma_c:=\{H^{-1}(c)\}$, $c\in (1-\epsilon,1+\epsilon)$ has an induced $b$-contact form whose associated Reeb $b$-vector field has neither periodic nor singular periodic orbits outside of its critical set $Z$.
\end{thm}

The above theorem thus shows that the Hamiltonian Seifert conjecture does not hold in the singular setting.

At the end of the article, we discuss a generalized version of the singular Weinstein conjecture. While our constructions do not exhibit singular periodic orbits, they do exhibit generalized singular periodic orbits, and in fact, infinitely many of them. In the topological quotient $M/Z$ these orbits (again, together with the singular point) become closed. The natural question, then, is if the following is true.

\begin{conjecture*}[Generalized Weinstein conjecture, Conjecture \ref{conj:gwc}]
  Every Reeb vector field of a (eventually singular) contact manifold admits at least one periodic orbit (away from the critical set $Z$) or a generalized singular periodic orbit.
\end{conjecture*}

\subsection*{Organization of the paper}
Section \ref{sec:review} contains a review of the basic notions of $b$-contact geometry, as well as a quick guide through the contemporary literature concerning dynamics of Reeb $b$-vector fields. In Section \ref{sec:tools}, we develop the main tools of this article. Section \ref{sec:counterexample} contains a counterexample to the singular Weinstein conjecture, and Section \ref{sec:bSeifert} contains a counterexample to the singular Hamiltonian Seifert conjecture. We conclude the article with a formulation of the \emph{generalized Weinstein conjecture}.

\subsection*{Acknowledgments} 
We thank Leonid Polterovich for formulating the generalized Weinstein conjecture that we describe in the last section of this article.

\section{Contact and $b$-contact geometry and the state of art on the singular Weinstein conjecture}\label{sec:review}

In this section, we give a brief introduction to $b$-contact geometry. Melrose introduced in \cite{M93} the notion of $b$-manifolds, which later appeared in the work of Nest and Tsygan in the context of deformation quantization on manifolds with boundary \cite{NT96}. The systematic study in the symplectic set-up was carried out in \cite{GMP14}. The language of $b$-manifolds is useful whenever one encounters logarithmic singularities in differential forms along a smooth hypersurface, which could be a boundary. This is particularly interesting as $b^m$-symplectic and $b^m$-contact forms arise when applying the McGehee transformation to study the behavior near infinity of the restricted planar three-body problem, as shown in \cite{invitation, MO21}. Consequently, the study of dynamical properties may lead to a better understanding of the dynamics of these classical problems in celestial mechanics.

A $b$-manifold is a smooth manifold $M$ with an {embedded} smooth hypersurface $Z$, which we shall call the \emph{critical set}. The \emph{$b$-tangent bundle} ${}^{b}TM$ is defined as the vector bundle whose sections are vector fields tangent to $Z$, and call these sections \textit{$b$-vector fields}. By taking the dual of ${}^{b}TM$, we obtain the $b$-cotangent bundle ${}^bT^*M$, which allows us to define $b$-forms of degree $k$ as sections $\omega\in\Gamma(\bigwedge^k({}^bT^*M)):={}^b\Omega^k(M)$. 

Around the critical set $Z$, given a defining function $z$ of $Z$, a $b$-form $\omega$ of degree $k$ decomposes as follows:
\begin{equation}\label{eq:decom}
    \omega = \alpha\wedge\frac{dz}{z} + \beta,
\end{equation}
with $\alpha\in\Omega^{k-1}(M)$ and $\beta\in\Omega^k(M)$.  This illustrates how the framework of $b$-manifolds is useful in dealing with logarithmic singularities along smooth hypersurfaces in differential forms. Furthermore, with this decomposition, one can extend the exterior derivative to $b$-forms:
\begin{equation*}
    d\omega = d\alpha\wedge\frac{dz}{z} + d\beta.
\end{equation*}
We refer the reader to \cite{GMP14} for more details on $b$-manifolds and proofs that all of these objects are well-defined. Equipped with this geometric structure, a $b$-contact form on a $b$-manifold is defined as follows.

\begin{definition}\label{def:b-structures}

        A $b$-contact form is a $b$-form of degree one on an odd-dimensional $b$-manifold $\alpha\in{}^b\Omega^1(M^{2n+1})$ such that $\alpha\wedge (d\alpha)^n$ is non-vanishing as a section of $\bigwedge^{2n+1}({}^bT^*M)$\textcolor{black}{, meaning that this defines a $b$-volume form}. Its kernel $\ker \alpha\subset {}^bTM$ is called a $b$-\emph{contact structure} and the associated \emph{Reeb $b$-vector field} is the unique $b$-vector field $R_\alpha$ such that
        \begin{equation*}
            \begin{cases}
                \iota_{R_\alpha} d\alpha = 0 \\
                \iota_{R_\alpha} \alpha = 1.
            \end{cases}
        \end{equation*}
\end{definition}

In \cite{MO18}, the geometry of these structures is studied, and the following important result concerning the dynamics of the Reeb $b$-vector field on the critical set is shown.

\begin{proposition}[\cite{MO18}]\label{prop:reebatz}
    Let $(M,Z,\alpha)$ be a $b$-contact manifold of dimension $3$, and write $\alpha = f\frac{dz}{z} + \beta$ with $f\in C^\infty(M)$ and $\beta\in\Omega^1(M)$. Then the restriction on $Z$ of the $2$-form $\omega = fd\beta + \beta\wedge df$ is symplectic and the Reeb $b$-vector field $R_\alpha$ is Hamiltonian on $Z$ with respect to $\omega$ with Hamiltonian function $-f|_Z$, i.e. $\iota_{R_\alpha}\omega = df$. The Hamiltonian $-f|_Z$ is called the \emph{exceptional Hamiltonian} associated with $\alpha$.
\end{proposition}

The authors observe in the sequel \cite[Proposition 6.1]{MO21} that this implies that there are always infinitely many periodic orbits on the critical set. Understanding the Reeb dynamics and the existence of periodic orbits outside of the critical set turns out to be much more intriguing: for instance, it is shown that the unit sphere in the standard $b$-symplectic $\mathbb{R}^{4}$ has an induced $b$-contact form, whose Reeb $b$-vector field does not have any periodic orbits outside of the critical set (see also Lemma \ref{lemma:singularbubbleform}, where a variation of this example is computed).

Due to the lack of classical dynamical invariants outside of the critical set, the authors introduced so-called singular periodic orbits and conjectured that this dynamical feature is a topological invariant on compact $b$-contact manifolds.

\begin{definition}\label{def:spo}
    Let $(M,Z,\alpha)$ be a $b$-contact manifold. A \textit{singular periodic orbit} is an integral curve $\gamma:\mathbb{R}\rightarrow M\setminus Z$ of the Reeb $b$-vector field such that $\lim_{t\rightarrow \pm\infty} \gamma(t) = p_\pm\in Z$ (see Figure \ref{fig:spodef}).
\end{definition}

The conjecture, as stated in \cite{MO21} claims the following:

\begin{conjecture}[The singular Weinstein conjecture, \cite{MO21}]
    Let $(M, Z, \alpha)$ be a compact $b^m$-contact manifold. Then there exists at least one singular periodic orbit or one periodic orbit outside of the critical set.
\end{conjecture}

The first steps towards this conjecture were obtained in \cite{MOP22, FMOP23}, where a semi-local version of this conjecture was considered (Theorem \ref{thm:2norinfty}). 

\begin{definition}
    An \textit{escape orbit} is an integral curve $\gamma:\mathbb{R}\rightarrow M\setminus Z$ of the Reeb $b$-vector field such that at least one of the semi-orbits has a stationary limit point on $Z$ (see Figure \ref{fig:spodef}).
\end{definition}

\begin{figure}[ht]
    \centering
        \input{Fig_SPO}
        \caption{A \textit{singular periodic orbit} whose $\omega-$ (respectively $\alpha-$) limit is the northpole (respectively southpole) and an \textit{escape orbit} whose $\omega$-limit is the northpole.}
        \label{fig:spodef}
\end{figure}
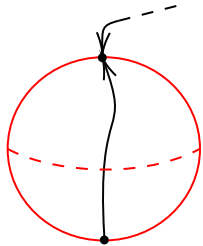

Escape orbits can thus be understood as a semi-local version around the critical set of singular periodic orbits (i.e., every singular periodic orbit is an escape orbit; the opposite does not hold as two escape orbits do not necessarily close up to form a singular periodic orbit).

\begin{theorem}[\cite{MOP22,FMOP23}]\label{thm:2norinfty}
Let $Z$ be a compact embedded surface in a $3$-dimensional manifold $M$. Then for a generic\footnote{Here the definition of \emph{genericity} is different in \cite{MOP22} and \cite{FMOP23}, see the respective references for the precise definition.} $b$-contact form having $Z$ as critical set, the associated Reeb $b$-vector field has at least $2N$ escape orbits, and infinitely many if the first Betti number of $Z$ is positive.
\end{theorem}

The escape to ``infinity" (the critical set) of a singular periodic orbit follows a distinctive pattern, where both its $\omega$- and $\alpha$-limit sets converge to a single point $p_\pm\in Z$. Since the limit sets of a vector field remain invariant, the points $p_\pm$ must be zeros of the field. Besides the notions of escape and singular periodic orbits, it is useful to consider generalized escape orbits and generalized singular periodic orbits. Such orbits appear naturally in several problems in celestial mechanics including that of Sitknikov \cite{sitnikov}. 

\begin{definition}\label{def:gspo}
A \textit{generalized escape orbit} is an integral curve $\gamma:\mathbb{R}\rightarrow M\setminus Z$ of the Reeb $b$-vector field such that there exist times $t_1<t_2<\cdots<t_k\to \infty$ or $t_{-1}>t_{-2}>\cdots >t_{-k}\to -\infty$ satisfying $\gamma(t_{\pm k})\to p\in Z$. Notably, $p$ need not be a zero of the field. Alternatively, $\gamma$ qualifies as a \textit{generalized singular periodic orbit} if \textit{both} its $\alpha$- and $\omega$-limit sets are \textit{contained in} the critical set $Z$.
\end{definition}

Specific instances of generalized escape periodic orbits, such as \emph{oscillatory motions} detailed in \cite{sitnikov,llibre,GMS}  have garnered significant attention. Figure \ref{fig:gardenoforbits} below shows several (generalized) singular periodic orbits and escape orbits.

\begin{figure}[H]
\begin{center}
\input{Fig_GardenOfOrbits}
\end{center}
 \caption{From left to right, an escape orbit, a singular periodic orbit, a generalized singular periodic orbit, and a generalized escape orbit. The critical set is the union of the two spheres colored red, and the red orbits are periodic orbits on the critical set.}\label{fig:gardenoforbits}
    
\end{figure}
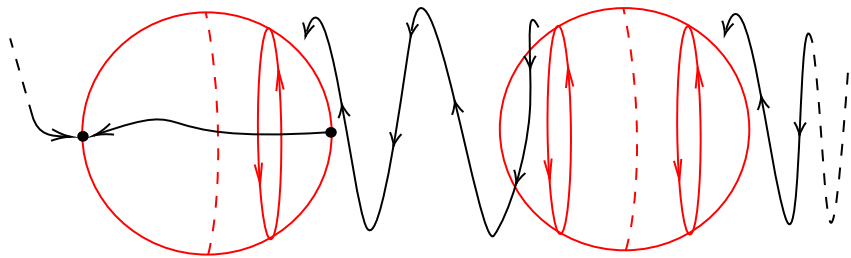


\section{$b$-contact manifolds with a prescribed number of singular periodic orbits}\label{sec:tools}

In this section, we develop tools to construct $b$-contact forms with prescribed dynamical properties in dimension $3$. More specifically, we show that we can modify $3$-dimensional (smooth) contact manifolds by adding closed singular hypersurfaces, in such a way that the Reeb $b$-vector field has a prescribed number of singular periodic orbits (see Definition \ref{def:spo}). The tools developed here, particularly Proposition \ref{prop:breakspo}, are used later in Sections \ref{sec:counterexample} and \ref{sec:bSeifert} to construct the counterexamples to the singular Weinstein and Seifert conjectures.


We structure this section in three parts. First, we show that there is a $b$-contact form on an open $3$-ball (with critical set given by a $2$-sphere) with exactly one singular periodic orbit, and we show that we can extend its contact structure with the standard contact structure outside of a sphere centered at the origin. We call these $3
$-balls \textit{singular bubbles} (the reason for the name will become clear shortly). In the second part, we show how we can replace Darboux neighborhoods in contact manifold with these `singular bubbles' to add singular periodic orbits to the flow. Finally, we conclude by detailing how to break off any undesired singular periodic orbits that add to the total count. We thus obtain the following. 

\begin{theorem}\label{thm:nspo}
Let $(M,\xi)$ be a coorientable $3$-dimensional contact manifold. Then for any integers $0\leq k\leq N$, there exists a $b$-contact form on $M$ whose critical set $Z$ consists of $N$ components diffeomorphic to $\mathbb S^2$ and the number of singular periodic orbits is exactly $k$. The associated $b$-contact structure coincides with $\xi$ outside a neighborhood of the balls enclosed by $Z$. Furthermore, there is an infinite number of generalized escape orbits converging to each component of $Z$ (i.e., orbits whose $\alpha$- or $\omega$-limit sets are on $Z$ but are not singular points).
\end{theorem}

\subsection{Step~1: Constructing Singular Bubbles}

First, we construct the basic building block for the proof of Theorem \ref{thm:nspo}, which provides an explicit $b$-contact form on an open $3$-ball with exactly one singular periodic orbit. The correspondence between Reeb $b$-vector fields and Beltrami $b$-vector fields in hydrodynamics (see~\cite{CMP19}) suggests calling these $b$-balls \textit{singular bubbles} since one can visualize our construction as placing fixed bubbles in water so that the flow is tangent to them. The aforementioned $b$-contact form is given by the following lemma. In the statement, we use the radial coordinate  $r^2=x^2+y^2+z^2$.

\begin{lemma}\label{lemma:singularbubbleform}
    Consider the $b$-manifold $(\mathbb{R}^3,\mathbb{S}^2 = $\{r=1\}$)$  endowed with the $b$-form
    \begin{equation*}
        \alpha = z(3 + r^2)\frac{rdr}{r^2 - 1} - \frac{r^2 + 1}{2}dz + x dy - y dx\,.
    \end{equation*}
   Then $\alpha$ is $b$-contact and its associated Reeb $b$-vector field has exactly one singular periodic orbit. The associated $b$-contact structure is denoted by $\xi_{b}$.
\end{lemma}

\begin{proof}
A straightforward computation yields
    \begin{equation*}
        d\alpha = 2(r^2 + 1)dz\wedge \frac{rdr}{r^2 - 1} + 2dx\wedge dy,
    \end{equation*}
    and
   \begin{equation*}
        \alpha\wedge d\alpha = 2\left( 4z^2 + (1+r^2)^2 \right)\frac{dx\wedge dy\wedge dz}{(r^2 - 1)},
    \end{equation*}
    which is, indeed, a $b$-volume form on $(\mathbb{R}^3,\mathbb{S}^2)$. This shows that $\alpha$ is $b$-contact.
    
Next, the associated Reeb $b$-vector field can be checked to be
    \begin{equation*}
        R_\alpha = \mu\left(-y\partial_x + x\partial_y + \frac{r^2 - 1}{r^2 + 1}\partial_z\right),
    \end{equation*}
    where $\mu := \frac{2(r^2 + 1)}{4z^2 + (r^2 + 1)^2}$. 
     Noticing that the function $x^2+y^2$ is a first integral of $R_\alpha$, it readily follows that this vector field, depicted in Figure \ref{fig:bReebpict} below, has only one singular periodic orbit, which is contained in the ball $\{r<1\}$ and goes from the north pole to the south pole. In the exterior region, there are two escape orbits associated to the north and south poles, respectively, but they are not singular periodic because they do not connect two singular points on the critical hypersurface. This completes the proof of the lemma.
    \end{proof}

\begin{figure}[h]
        \centering
        \input{Fig_bReeb}
        \caption{Sketch of the Reeb $b$-vector field in Lemma~\ref{lemma:singularbubbleform}. The singular periodic orbit is colored red. A helical orbit converging to a periodic orbit on the critical set is also shown (which is a \textit{generalized singular periodic orbit}), and two escape orbits in the complement of the ball bounded by the critical surface.}
        \label{fig:bReebpict}
    \end{figure}
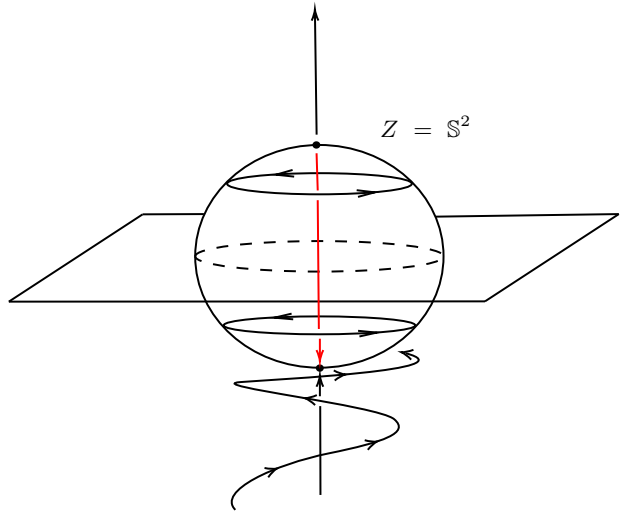

\begin{remark}
The $1$-form $\alpha$ in Lemma \ref{lemma:singularbubbleform} was found by pulling back the $b$-contact form constructed in \cite[Example 6.7]{MO21} from $\mathbb{S}^3$ onto $\mathbb R^3$  using a stereographic projection, and then multiplying it by a global factor. A similar construction was used in~\cite{MO21} to show the existence of \textit{traps} in $b$-Reeb dynamics.
\end{remark}

\begin{remark}\label{rem:generalizedescape}
An additional interesting feature of the Reeb $b$-vector field $R_\alpha$ in Lemma~\ref{lemma:singularbubbleform} is that it exhibits an infinite number of \textit{generalized escape orbits}. These are all the helical orbits converging to parallels on the critical sphere. Generalized periodic orbits appear in celestial mechanics,  for instance in the work of Sitnikov~\cite{sitnikov}, see also Section 5 in \cite{MOP22}.
\end{remark}

In what follows let us denote by $\xi_b$ the $b$-contact structure defined by the $b$-contact form of Lemma \ref{lemma:singularbubbleform}. We claim that we can cut out a ball from $\mathbb{R}^3$ with the standard contact structure $\xi_{st}$ and glue the singular bubble with the $b$-contact structure $\xi_b$ back in along an \textit{open shell} (i.e., a domain in $\mathbb R^3$ diffeomorphic to a spherical annulus $\mathbb S^2\times (-\varepsilon,\varepsilon)$). The process is illustrated in Figure \ref{fig:extensionpict} below. This allows us to insert these singular bubbles into any contact manifold by cutting and pasting them into contact Darboux neighborhoods of said manifold.

 \begin{figure}[h]
        \centering
        \input{Fig_extendwithstandard}
        \caption{Gluing the singular bubble with the standard contact structure of $\mathbb{R}^3$ along an open shell (the area within the dashed annulus). The critical surface is colored red and is not part of the shell.}
        \label{fig:extensionpict}
    \end{figure}
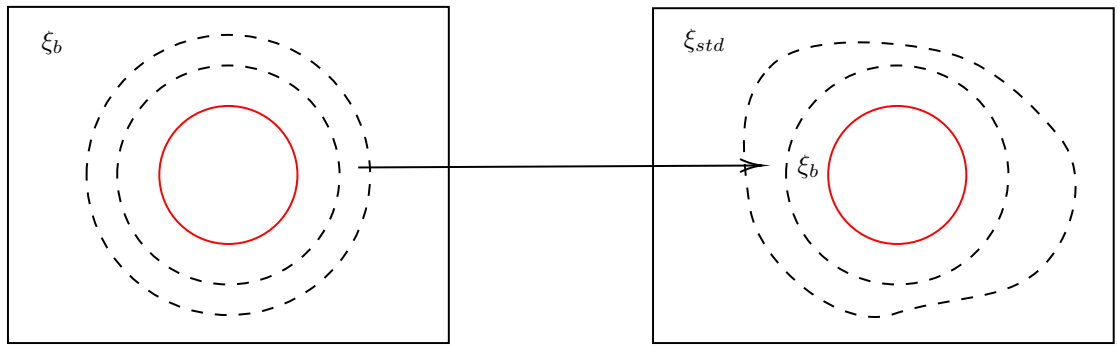

We observe that, although the $b$-contact structure $\xi_b$ is a contact structure in the usual (i.e., not $b$-) sense on an open shell away from the ball $\{r \leq 1\}$, we cannot easily resort to the Gray stability theorem to conclude that it is contactomorphic to an open shell with $\xi_{st}$, because shells are open manifolds. We therefore use the following theorem by Giroux \cite[Section II.1.2]{G91} (see also \cite[Theorem~2.5.22]{G08}). 

\begin{theorem}[Giroux]\label{thm:giroux}
    Let $S_i$ be closed surfaces in contact $3$–manifolds $(M_i, \xi_i)$, $i = 0, 1$, and $\phi: S_0 \rightarrow S_1$ a diffeomorphism preserving the characteristic foliations. Then there is a contactomorphism $\psi: (\mathcal{N}(S_0),S_0) \rightarrow (\mathcal{N}(S_1),S_1)$ of suitable neighbourhoods $\mathcal{N}(S_i)$ of $S_i$.
\end{theorem}

Recall that the \textit{characteristic foliation} $S_\xi$ of a surface $S$ in a $3$-dimensional contact manifold $(M, \xi)$ is the singular 1–dimensional foliation of $S$ defined by the distribution $TS \cap \xi|_S$.

In the context of our construction, we apply Giroux's theorem choosing the closed surfaces to be spheres of radius $R>1$, and their neighborhoods are the open shells we want to identify (in fact, due to a minor technicality, the chosen surfaces will be \textit{$C^0$-close} to said spheres, not the spheres exactly). The gluing construction is stated in the following lemma.

\begin{lemma}\label{lemma:shells}
    For $R>1$ there are open shells near $\mathbb{S}^2_R$ (the 2-sphere of radius $R$) in $(\mathbb{R}^3,\xi_{st})$ and $(\mathbb{R}^3,\xi_{b})$ respectively, which are contactomorphic.
\end{lemma}
\begin{proof}
    In view of Giroux's theorem, it would suffice to prove that the characteristic foliations induced by $(\mathbb{R}^3,\xi_{st})$ and $(\mathbb{R}^3,\xi_{b})$ on $\mathbb{S}^2_R$ are diffeomorphic. As it turns out, the characteristic foliations are only \textit{homeomorphic}, but as we shall see, we can perturb the spheres slightly to make the resulting characteristic foliations diffeomorphic and then apply Giroux's theorem.
    
    We denote the characteristic foliations by $S_{\xi_{st}}$ and $S_{\xi_b}$. For the standard contact structure, we take the kernel of the 1-form $\alpha_{st} = dz + xdy - ydx$. The characteristic foliation on a sphere $\mathbb{S}^2_R$ is computed in \cite[Example 2.5.19]{G08}, and it is generated by the integral curves of the vector field
    \begin{equation*}
        (xz - y)\partial_x + (yz + x)\partial_y - (x^2 + y^2)\partial_z.
    \end{equation*}
    In particular, the only singular points are the north and south poles, which are non-degenerate, and all the integral curves are transverse to the level sets $\{z = const\}$. A straightforward computation shows that the characteristic foliation induced by $(\mathbb{R}^3,\xi_b)$ on $\mathbb{S}^2_R$ is generated by the integral curves of the vector field
    \begin{equation*}
        (xz + \frac{r^2 + 1}{2}y)\partial_x + (yz - \frac{r^2 + 1}{2}x)\partial_y - (x^2 + y^2)\partial_z.
    \end{equation*}

    Indeed, the characteristic foliation is given by $S_{\xi_b}=T_p\mathbb{S}^2_R\cap (\xi_b)_p$ for each $p\in \mathbb{S}^{2}_R$, which can be translated into the following system of equations,
    \begin{equation*}
    \begin{cases}
         -\frac{r^2+1}{2}X_z+xX_y-yX_x=0\\
        xX_x+yX_y+zX_z=0\,,
    \end{cases}
    \end{equation*}
    for the unknown vector field $X_x\partial_x+X_y\partial_y+X_z\partial_z$ that generates the characteristic foliation $S_{\xi_b}$. This system can be solved to obtain the above-mentioned result.
    
    The singular points of this foliation are also the north and south poles, and are of elliptic type (in the sense that the eigenvalues of the linearization have two eigenvalues with real parts of equal sign, \cite[p. 166]{G08}); they are again non-degenerate, and its orbits are also transverse to the level sets $\{z = const\}$. Therefore, we can construct a map between the characteristic foliations by mapping the flows of the vector fields above from $\{z = 0\}$ onto each other and maintaining the north and south poles fixed. This mapping is in fact a homeomorphism, \textit{not} necessarily a diffeomorphism. However, this inconvenience is easily overcome with the following proposition, again from \cite[Proposition 4.8.14]{G08}.

    \begin{proposition}
        Let $j_i:S\rightarrow (M_i,\xi_i)$, $i=0,1$, be two embeddings of a closed oriented surface $S$ into $3$-dimensional contact manifolds such that the characteristic foliations $(j_i(S))_{\xi_i}$ are of Morse-Smale type and homeomorphic (as oriented singular foliations) via an orientation-preserving homeomorphism that respects the signs of the singular points (this is an extra condition for the saddle points only). Then there is an embedding $j_1'$, $C^0$-close to $j_1$, such that $(j_0(S))_{\xi_0}$ and $(j_1'(S))_{\xi_1}$ are diffeomorphic.
    \end{proposition}
    
    The condition that a characteristic foliation is of Morse-Smale type is a generic one that both foliations satisfy.\footnote{For a characteristic foliation to be of Morse-Smale type there must be finitely many singularities and closed orbits, all of which must be non-degenerate; the $\alpha$- and $\omega$-limit sets of all flowlines must be singular points or closed orbits; and there are no trajectories connecting saddle points. See \cite[Definition 4.6.8]{G08}.}
    The sign condition does not apply to the present case because there are no saddle points in the characteristic foliations.
    Thus, we can perturb the $\mathbb{S}^{2}_R$ in $(\mathbb{R}^{3},\xi_{st})$ to a $C^0$-close spherical surface  whose characteristic foliation $S_{\xi_{st}}'$ is diffeomorphic to $S_{\xi_b}$. An immediate application of Giroux's Theorem \ref{thm:giroux} yields the desired contactomorphism between open shells near the 2-spheres of radius $R$.
\end{proof}

The idea that we shall exploit in the second step is that this lemma allows us to replace contact Darboux $3$-balls with the singular bubbles of Lemma \ref{lemma:singularbubbleform}.

\subsection{Step 2: Transplanting singular bubbles into contact manifolds}

In this subsection, we show how we can place the singular bubbles into a given contact manifold and extend the $b$-contact form of Lemma \ref{lemma:singularbubbleform} to a globally defined $b$-contact form with a prescribed critical set. This is the content of the following proposition:

\begin{proposition}\label{prop:addsingularbubbles}
    Let $(M,\xi)$ be a coorientable contact $3$-manifold. Then for any positive integer $N$, there exists a $b$-contact form $\tilde\alpha$ on $M$ such that:
    \begin{itemize}
        \item The critical set $Z$ consists of $N$ components diffeomorphic to $\mathbb S^2$.
        \item The number of singular periodic orbits is \emph{at least} $N$.
        \item  There are finitely many escape orbits.
    \end{itemize}
    The $b$-contact structure defined by $\tilde\alpha$ coincides with $\xi$ outside a neighborhood of the balls bounded by $Z$. Furthermore, each component of $Z$ has an infinite number of generalized escape orbits converging to it.
\end{proposition}

\begin{proof}
   Let $p_1,\dots,p_N$ be $N$ different points in $M$. By Darboux's Theorem, we can take disjoint neighborhoods $U_i$ of $p_i$ which are contactomorphic to $(\mathbb{R}^3,\xi_{st})$. We then cut out $3$-balls within these neighborhoods and glue back the singular bubbles with the $b$-contact structure of Lemma \ref{lemma:singularbubbleform} along the open shells of Lemma \ref{lemma:shells}. This produces a $b$-contact structure $(M',\xi_b)$ which is contactomorphic to $(M,\xi)$, with the diffeomorphism being the identity outside the contractible neighborhoods $U_i$ in $M$ and $\hat{U}_i$ in $M'$. We denote the shells along which we perform surgery by $S_i$, and by abuse of notation, we denote the new,  manifold after surgery by $M$ instead of $M'$.
   Thus we have obtained a $b$-contact structure on $M$ which has as a critical set $Z$ consisting of $N$ spheres. We will show that we can choose a $b$-contact form $\alpha_b$ defining $\xi_b$ whose associated Reeb vector field satisfies the statement of the proposition. 
   Thus, all that remains is to construct a global $b$-contact form whose Reeb $b$-vector field has $N$ singular periodic orbits.

    Let $\hat{\alpha}$ be a contact form defining the contact structure $\xi$, and $\alpha_i$ the $b$-contact forms of Lemma \ref{lemma:singularbubbleform} extended to the modified neighborhoods $\hat{U_i}$ so that the contact structures that they define coincide with $\xi$ along the open shells $S_i$. Since $\alpha_i$ and $\hat\alpha$ define the same contact structures on and outside the shells, they are related by positive smooth functions $f_i$. That is, there are $f_i>0$ defined on $\overline S_i$ such that $\alpha_i|_{S_i} = f_i\hat{\alpha}|_{S_i}$, as shown in Figure \ref{fig:extension}.
    \begin{figure}[h]
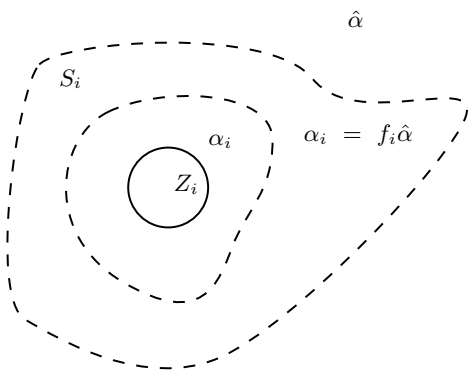

        \centering
        \include{Fig_extension}
        \caption{Gluing in the $b$-contact forms.}
        \label{fig:extension}
    \end{figure}
    We now take step functions $g_i\in C^\infty(M,[0,1])$ such that $g_i = 0$ outside the outer boundary of $S_i$ and $g_i = 1$ inside the inner boundary of $S_i$ (which contains the singular bubble), and define
    \begin{equation}\label{eq:bformbubble}
        \tilde\alpha := \hat{\alpha}(1 - \sum_i g_i) + \sum_i g_i \alpha_i.
    \end{equation}
    Obviously, $\tilde\alpha=\hat\alpha$ outside the outer boundary of the union of all $S_i$, and $\tilde\alpha=\alpha_i$ inside the inner boundary of each $S_i$. In the shell $S_i$, we have by construction
    \[
    \tilde\alpha=(1-g_i+f_ig_i)\hat\alpha\,.
    \]
    Noticing that $1-g_i+f_ig_i>0$ in $\overline S_i$ for all $i$, we conclude that $\tilde\alpha$ is a global $b$-contact form that extends the $\alpha_i$ outside of the singular bubbles. 
    
    The Reeb $b$-vector field $\tilde R_{\alpha}$ associated to $\tilde\alpha$ has at least $N$ singular periodic orbits (one within each bubble). Moreover, outside the bubbles, there are at most $2N$ escape orbits, which may join to produce additional singular periodic orbits. Finally, as observed in Remark \ref{rem:generalizedescape}, the Reeb $b$-vector field also exhibits an infinite amount of generalized escape orbits converging to each critical sphere.
\end{proof}

\subsection{Step~3: Breaking undesired singular periodic orbits}

In this final step, we show that we can perturb a $b$-contact form with finitely many escape orbits to \emph{break} undesired singular periodic orbits. This is the content of the following proposition, which actually only uses properties of smooth contact forms. 

\begin{proposition}\label{prop:breakspo}
    Let $(M,Z)$ be a $3$-dimensional $b$-manifold with $b$-contact form $\alpha$, and assume that the corresponding Reeb $b$-field $R_\alpha$ has finitely many escape orbits and $N$ singular periodic orbits. Then for every $0\leq k\leq N$ there is another $b$-contact form $\overline{\alpha}$ that coincides with $\alpha$ except on a small compact domain $K\subset M\backslash Z$, and its associated Reeb $b$-vector field $ R_{\overline{\alpha}}$ has exactly $k$ singular periodic orbits.  Moreover, $\overline\alpha$ can be made arbitrarily $C^\infty$-close to $\alpha$ on $K$.
\end{proposition}

\begin{proof}
    We first show how to break \textit{one} singular periodic orbit with a perturbation supported in the complement of $Z$. Let $p\in M\backslash Z$ be a point on a singular periodic orbit, and $U$ a Darboux neighborhood containing $p$ and intersecting no other escape orbits. Such a neighborhood exists because $R_\alpha$ has finitely many escape orbits. We assume that $p$ is sufficiently close to the origin but \textit{not} on the $z$-axis of this chart. Endowing $U$ with cylindrical coordinates (and after rescaling the $z$ direction), the contact form $\alpha|_U$ has the expression
    \begin{equation*}
        \alpha|_U = \frac{1}{2}dz + r^2d\varphi\,,
    \end{equation*}
and its associated Reeb vector field in this chart is $R_\alpha = 2\partial_z$, so all orbits are vertical lines (we refer to $\alpha|_{U}$ as simply $\alpha$ to simplify notation). In particular, the singular periodic orbit passing through $p$ is a vertical line in $U$ that does not contain the origin. Now consider another contact form on $U$,
\begin{equation}
    \alpha' = zrdr + \frac{1}{2}(1 + z^2 - r^2)dz + r^2d\varphi\,,
\end{equation}
whose Reeb vector field is proportional to $\partial_z + \partial_\varphi$. We will use this contact form to perturb $\alpha$. The result of this perturbation is shown in Figure \ref{fig:breakingspo}. 

To this end, it is convenient to introduce a compact set $K\subset U$, with $p\in K$, of the form $D_\delta\times [-\delta,\delta]$ in cylindrical coordinates, where $D_\delta$ is a closed $2$-disk of radius $\delta>0$. Take a smooth bump function $f:\mathbb{R}\rightarrow [0,1]$ which is equal to $1$ in $I_{ct}:=[-\delta/2,\delta/2]$, and whose support is contained in $(-\delta,\delta)$. Without loss of generality, we assume that the singular periodic orbit is a vertical line whose $r$-coordinate is in the interior of the interval $I_{ct}$. For any small $\varepsilon>0$,  consider the $1$-form
\begin{equation*}
    \tilde{\alpha} = (1-\varepsilon f(r)f(z))\alpha + \varepsilon f(r)f(z)\alpha',
\end{equation*}
defined on $M$, which coincides with $\alpha$ on $M\backslash K$, and is $C^\infty$-close to $\alpha$ on $K$ provided that $\varepsilon$ is small enough. Obviously, $\tilde\alpha$ is a $b$-contact form on $M$. 

We claim that $\tilde{\alpha}$ has one singular periodic orbit less than $\alpha$. To show this, we compute the Reeb field $ R_{\tilde \alpha}$. Since we are only interested in the singular periodic orbit, we can restrict our study to the set $D_{\delta/2}\times (-\delta,\delta)\subset K$, which contains the point $p$. A straightforward computation shows that on the set $D_{\delta/2}\times (-\delta,\delta)$ we have
\begin{equation*}
    \tilde{\alpha} =  \varepsilon f(z)zrdr + \frac{1}{2}(1 + \varepsilon f(z)(z^2 - r^2))dz + r^2d\varphi, 
\end{equation*}
and
\begin{align*}
    d\tilde{\alpha} &= \varepsilon(f'z + f)rdz\wedge dr - \varepsilon frdr\wedge dz + 2rdr\wedge d\varphi =\\
    &= \varepsilon(f'z + 2f)rdz\wedge dr + 2rdr\wedge d\varphi.
\end{align*}
To find the corresponding Reeb vector field $ R_{\tilde{\alpha}} = \tilde R^r\partial_r + \tilde R^\varphi \partial_\varphi + \tilde R^z \partial_z$, we see immediately that $\tilde R^r = 0$, and that $\tilde R^\varphi$ and $\tilde R^z$ satisfy
\begin{equation*}
    \varepsilon(\frac{1}{2}f'z + f)\tilde R^z = \tilde R^\varphi.
\end{equation*}
If $\varepsilon$ is small enough, we can compute the vector field $ R_{\tilde \alpha}$ in terms of an $\varepsilon$-expansion, which yields:
\[
\tilde R^r=0\,,\qquad \tilde R^\varphi= \varepsilon(f'(z)z + 2f(z))+O(\varepsilon^2)\,,\qquad \tilde R^z=2+O(\varepsilon)\,.
\]
Therefore, we can integrate approximately the integral curves of $ R_{\tilde \alpha}$. In the particular case of the trajectory that corresponds to the singular periodic orbit of $R_\alpha$, taking as initial condition the point $r_0=\delta_0<\delta/2$, $\varphi_0=0$, $z_0=-\delta$, we obtain
\[
r(t)=r_0\,,\qquad \varphi(t)=\frac12\varepsilon\int_{-\delta}^{-\delta+2t}\Big(f'(s)s+2f(s)\Big)ds + O(\varepsilon^2 t)\,, \qquad z(t)=-\delta + (2+O(\varepsilon))t\,.
\]
Since the time taken from the perturbed orbit to go from $\{z=-\delta\}$ to $\{z=\delta\}$ is $T=\delta + O(\varepsilon)$, we finally conclude that
\[
\varphi(T)=\frac12\varepsilon\int_{-\delta}^{\delta}\Big(f'(s)s+2f(s)\Big)ds + O(\varepsilon^2\delta)=\frac12\varepsilon\int_{-\delta}^\delta f(s)ds + O(\varepsilon^2\delta)\,,
\]
where we have integrated by parts and used that $f(z)=0$ near $z=\pm\delta$. It then follows from the fact that $f$ is a non-negative function that
\[
\varphi(T)=\frac{C}{2}\varepsilon\delta +O(\varepsilon^2\delta)>0,
\]
for some $C>0$, and hence the continuation of the singular periodic orbit along the flow of $ R_{\tilde{\alpha}}$ rotates slightly within a small cylindrical neighborhood of the origin, in addition to moving upwards (see Figure \ref{fig:breakingspo}). 

The upshot is that the semi-orbit of the singular periodic orbit coming into the cylindrical set $K$ from below no longer matches the semi-orbit coming out from above, and thus the singular periodic orbit is broken into two escape orbits. Furthermore, these semi-orbits do not coincide with any other semi-orbits of escape orbits, as the neighborhood we had taken had no other escape orbits. In summary, we have proved that the Reeb $b$-vector field $ R_{\tilde{\alpha}}$ has one singular periodic orbit less than $R_\alpha$. Iterating this process $N-k$ times with perturbations supported in pairwise disjoint compact sets $K_i$, we prove the statement of the proposition.

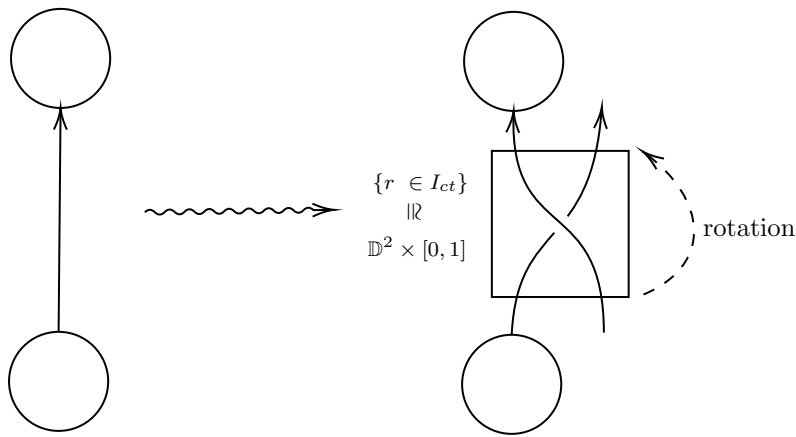
\begin{figure}[h]
    \centering
    \input{Fig_breakingspo}
    \caption{Illustration of the perturbation used to break singular periodic orbits. The circles represent the critical spheres. On the left, a singular periodic orbit is shown. On the right, it is broken into two escape orbits after a small perturbation of the $b$-contact form.}
    \label{fig:breakingspo}
\end{figure}
\end{proof}

\begin{remark}
The above proposition shows that the set of $b$-contact forms with no singular periodic orbits is dense (in the $C^\infty$-topology) in the space of $b$-contact forms with finitely many escape orbits.
\end{remark}

The proof of Theorem~\ref{thm:nspo} is then straightforward by combining Propositions~\ref{prop:addsingularbubbles} and~\ref{prop:breakspo}:

\begin{proof}[Proof of Theorem \ref{thm:nspo}]
Take the $3$-dimensional contact manifold $(M,\xi)$ and add in $N$ singular bubbles as in Proposition \ref{prop:addsingularbubbles}. This yields a $b$-contact form $\tilde\alpha$ with a critical set $Z$ that consists of $N$ copies of $\mathbb S^2$ and whose corresponding Reeb $b$-vector field has at least $N$ singular periodic orbits and finitely many escape orbits. We can then use Proposition~\ref{prop:breakspo} to construct a second $b$-contact form $\alpha$ so that exactly $k$ singular periodic orbits remain. Taking the support of the perturbation to be contained in a neighborhood of the bubbles, we finally infer that the associated $b$-contact structure coincides with $\xi$ outside a neighborhood of the balls enclosed by the critical set $Z$. The existence of infinitely many generalized escape orbits follows immediately from the properties of the bubbles. This completes the proof of the theorem.
\end{proof}

In the following section we use some of the techniques presented here to give a counterexample to the singular Weinstein conjecture as stated in \cite{MO21}. More precisely, we use Proposition~\ref{prop:breakspo} to break undesired singular periodic orbits.

\section{A counterexample to the singular Weinstein conjecture.}\label{sec:counterexample}

As explained in the introduction, since compactifying contact manifolds often leads to singular contact manifolds, it becomes important to study singular Reeb dynamics. In particular, it would be interesting to have a singular counterpart to the Weinstein conjecture. In \cite{MO21} a singular counterpart to the Weinstein conjecture for $b$-contact manifolds was put forth.

\begin{conjecture}\label{conj:singWeinst}
    Let $(M,Z,\alpha)$ be a $b$-contact manifold. Then there exists either a periodic orbit away from the critical set $Z$ or a singular periodic orbit.
\end{conjecture}

In this section, we use the results of Section~\ref{sec:tools} to construct a $b$-contact form on $(\mathbb{S}^3,\mathbb{S}^2)$ which has no singular periodic orbits and no periodic orbits away from the critical set, which is a counterexample to Conjecture \ref{conj:singWeinst}. This example is based on  \cite[Example~6.7]{MO21}, which is the singular counterpart to the Hopf vector field.  We call it therefore the \emph{Hopf $b$-vector field}, though the authors did not identify it as such. The Hopf $b$-vector field was already shown to have no periodic orbits away from the critical set $Z = \mathbb{S}^2$, and it is also easy to see that it has finitely many escape orbits. We then can use Proposition \ref{prop:breakspo} to break off all singular periodic orbits without creating any new periodic orbits.

\subsection{The Hopf $b$-vector field}\label{sec:bhopf}
In this subsection, we present and study the dynamics of the Hopf $b$-vector field. We follow \cite[Example~6.7]{MO21}. Consider the standard $b$-symplectic Euclidean space $(\mathbb{R}^4,\omega)$, with
\begin{equation*}
    \omega = \frac{dx_1}{x_1}\wedge dy_1 + dx_2\wedge dy_2.
\end{equation*}
We can construct a $b$-contact form $\alpha = \iota_X \omega$ on $(\mathbb{S}^3,\mathbb{S}^2)$ by contracting the $b$-symplectic form with a Liouville vector field $X$ transverse to $\mathbb{S}^3$ given by
\begin{equation*}
    X = \frac{1}{2}(x_1\partial_{x_1} + 2y_1\partial_{y_1} + x_2\partial_{x_2} + y_2\partial_{y_2}).
\end{equation*}
The result is
\begin{equation*}\label{eq:b-HopfContact}
    \alpha = \frac{1}{2}(dy_1 - 2y_1\frac{dx_1}{x_1} + x_2dy_2 - y_2dx_2).
\end{equation*}
The associated Reeb $b$-vector field is
\begin{equation}\label{eq:bReebWeinst}
    R_\alpha = \frac{2}{1+y_1^2}(-x_1y_1\partial_{x_1} + x_1^2\partial_{y_1} - y_2\partial_{x_2} + x_2\partial_{y_2}).
\end{equation}
Making abstraction of the conformal factor, it is obvious that the vector field 
$$X_H := -x_1y_1\partial_{x_1} + x_1^2\partial_{y_1} - y_2\partial_{x_2} + x_2\partial_{y_2}$$
defines the same orbits as the vector field $R_\alpha$.
The subscript $H$ notation is justified by the fact that the vector field $X_H$ can be interpreted as the Hamiltonian vector field of two uncoupled harmonic oscillators on the standard $b$-symplectic manifold.  More precisely, the vector field $X_H$ is the Hamiltonian vector field with respect to the Hamiltonian $H = \frac{x_1^2 + y_1^2}{2} + \frac{x_2^2 + y_2^2}{2}$ on the standard $b$-symplectic $\mathbb{R}^4$. This motivates calling $X_H$ the \textit{Hopf $b$-vector field} when we restrict it to the energy level $H = 1$ (for another perspective on the smooth harmonic oscillator in the context of $b$-integrable systems, see \cite{KM17}).

As noted in \cite{MO21}, it is immediately clear from the expression of $X_H$ that it has no periodic orbits away from $Z = \mathbb{S}^2$, since the $(x_1,y_1)$ system has no periodic orbits.  To further understand the flow, it is convenient to use the stereographic projection onto $\mathbb{R}^3$ from the north pole $(1,0,0,0)$ of $\mathbb S^3$. The stereographic projection is given by
\begin{align*}
\Psi: \mathbb{S}^3\setminus (1,0,0,0) &\rightarrow \mathbb{R}^3 \\
(x_1,y_1,x_2,y_2)&\mapsto (\frac{x_2}{1-x_1},\frac{y_2}{1-x_1},\frac{y_1}{1-x_1}).
\end{align*}
The stereographic projection of the usual Hopf vector field $-y_1\partial_{x_1} + x_1\partial_{y_1} - y_2\partial_{x_2} + x_2\partial_{y_2}$ can be written in cylindrical coordinates $(\rho,\phi,z)$ as 
\begin{equation}\label{eq:smoothHopf}
-z\rho\partial_\rho - \frac{1 - \rho^2 + z^2}{2}\partial_z + \partial_\phi.
\end{equation}
In contrast, the stereographic projection of the Hopf $b$-vector field $-x_1y_1\partial_{x_1} + x_1^2\partial_{y_1} - y_2\partial_{x_2} + x_2\partial_{y_2}$ is (again, in cylindrical coordinates, and letting $r^2 = x^2 + y^2 + z^2$),
\begin{equation}\label{eq:bHopffield}
    \frac{r^2-1}{r^2+1}\left( -z\rho\partial_\rho - \frac{1 - \rho^2 + z^2}{2}\partial_z \right) + \partial_\phi.
\end{equation}
Notice that the only difference between Equations \eqref{eq:smoothHopf} and \eqref{eq:bHopffield} is the $\frac{r^2-1}{r^2+1}$ factor, which multiplies the $\partial_\rho$ and $\partial_z$ components, but crucially not the $\partial_\phi$ component, of the Hopf $b$-vector field. Figure \ref{fig:bHopf} shows the flow of the Hopf $b$-vector field on the y = 0 plane. Notice the white circle representing stationary points at the critical surface ${r = 1}$. The points on this circle rotate
according to $\partial_\phi$ along the parallels of the critical $\mathbb{S}^{2}$.

\begin{figure}[h]
    \centering
    \includegraphics{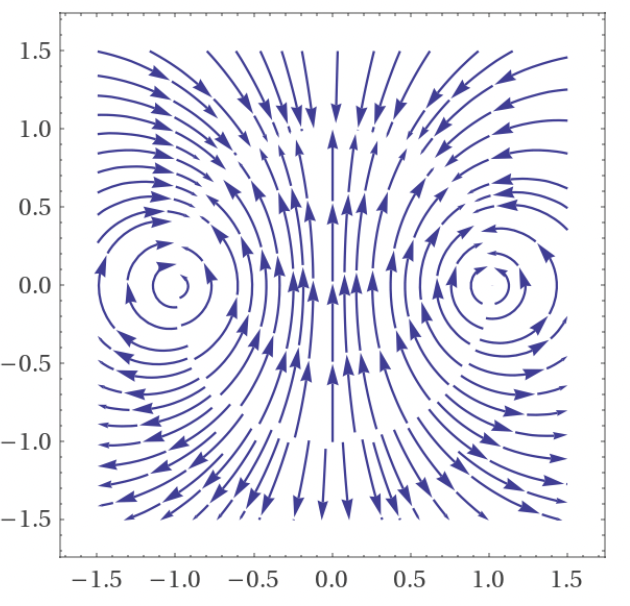}
    \caption{Illustration of the Hopf $b$-vector field on the $y=0$ section. It is essentially the smooth Hopf vector field rescaled as it approaches $r=1$.}
    \label{fig:bHopf}
\end{figure}
The smooth Hopf vector field has the interesting property that all orbits are periodic and pass through $r=1$. In the Hopf $b$-vector field flow, all of the periodic orbits are broken by the singular $\mathbb{S}^2$, and if the orbit along the smooth Hopf flow of a point $p\in\mathbb{S}^3$ intersects $\mathbb{S}^2$ at $z = \pm h$, then its orbit along the $b$-Hopf flow has $\alpha$- and $\omega$-limit sets $\mathbb{S}^2\cap\{z = \pm h\}$, because as it approaches $\mathbb{S}^2$, the rotation $\partial_\phi$ brings the orbit close to every point on $\mathbb{S}^2\cap\{z = \pm h\}$ infinitely often. Each orbit remains on the inside of, or the outside of (or on) the critical $\mathbb{S}^2$ for all time. In summary, the Hopf $b$-vector field has the following interesting properties.

\begin{proposition}\label{prop:bHopforbits}
    The orbits of the Hopf $b$-vector field on $\mathbb{S}^3$, and hence of the Reeb $b$-vector field $R_\alpha$ constructed before, are:
    \begin{enumerate}[a)]
        \item     Two singular periodic orbits (along the $z$ axis) and no other escape orbits.
        \item Two critical points on the north and south poles of the critical $\mathbb{S}^2$.
        \item Periodic orbits on all the parallels of the critical $\mathbb{S}^2$.
        \item Generalized escape orbits whose $\alpha$- and $\omega$-limit sets are the aforementioned periodic orbits on $\mathbb S^2$.
    \end{enumerate}
    
\end{proposition}

\begin{figure}[h]
    \centering
    \input{Fig_bHopf3D}
    \caption{Stereographic projection of some orbits of the Hopf $b$-vector field. The four types of orbits are depicted. The orbits exiting from the north pole and arriving at the south pole outside of the $\mathbb{S}^2$ are part of the same singular periodic orbit on $\mathbb{S}^3$, though in this figure they look like two separate escape orbits.}
    \label{fig:bHopf3D}
\end{figure}
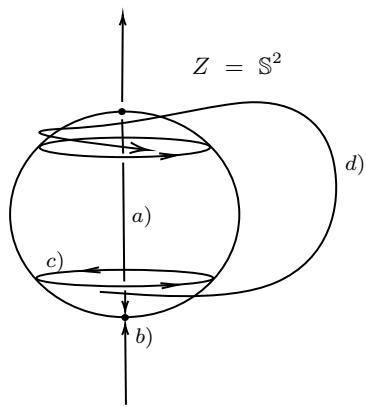

\subsection{The counterexample}\label{subsec:counterexample}
We now come back to the original goal of this section, which is to prove the following result.
\begin{theorem}\label{thm:counterexample}
There is a $b$-contact form on the $b$-manifold $(\mathbb S^3,\mathbb S^2)$ which has no singular periodic orbits and no periodic orbits away from the critical set.
\end{theorem}
\begin{proof}
    The Reeb $b$-vector field $R_\alpha$ on $(\mathbb S^3,\mathbb S^2)$ constructed in Section~\ref{sec:bhopf} has two singular periodic orbits, no other escape orbits, and no periodic orbits away from the critical surface, cf. Proposition~\ref{prop:bHopforbits}. We can then use Proposition~\ref{prop:breakspo} to construct a $b$-contact form $\alpha'$ on $(\mathbb S^3,\mathbb S^2)$ where the two singular periodic orbits \emph{break} (and no new ones arise). Furthermore, the Reeb $b$-vector field $R_{\alpha'}$ coincides with $R_\alpha$ in the complement of a small compact set away from the critical surface, so it follows that no periodic orbits arise away from the critical $\mathbb{S}^2$. The reason is that the $\alpha$- and $\omega$-limit sets of all orbits of $R_{\alpha}$ away from the critical surface are on the critical $\mathbb{S}^2$ by Proposition~\ref{prop:bHopforbits}, so the same holds true for the orbits of $R_{\alpha'}$. Figure \ref{fig:bHopfPerturbed} shows the orbits of the perturbed Hopf $b$-vector field. This concludes the construction of the counterexample.

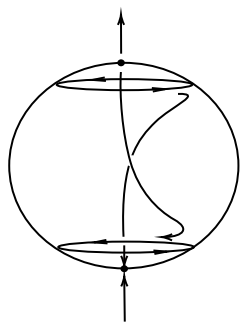
\begin{figure}[h]
    \centering
    \input{Fig_bHopfPerturbed}
    \caption{Orbits of the perturbed Hopf $b$-vector field. Note that the singular periodic orbit passing through the $z$ axis becomes a generalized escape orbit after the perturbation.}
    \label{fig:bHopfPerturbed}
\end{figure}
\end{proof}

\subsection{The singular Weinstein conjecture is generically true when the critical surface has a component of positive genus} In this final subsection we show that, in contrast with the case that all the critical surfaces are $\mathbb S^2$, when at least one of the components of $Z$ has positive genus, the associated Reeb $b$-vector field exhibits a singular periodic orbit almost surely.

\begin{proposition}\label{prop:genericSPOonZ}
 Let $(M,Z,\alpha)$ be a $b$-contact $3$-dimensional manifold. Assume that $Z$ has a component $Z_0$ with positive genus. Then for generic $\alpha$ the associated Reeb $b$-vector field has at least one singular periodic orbit on $Z_0$.  
\end{proposition}
 Generic means open and dense in the $C^\infty$ topology for $b$-contact forms, which is defined as follows:
 
 \begin{remark}
The $C^\infty$-topology on the space of $b$-forms is defined as follows. As $b$-forms are smooth differential forms away from $Z$, away from a tubular neighborhood of $Z$, the topology coincides with the usual one. Around the critical set, given a $b$-form $\omega$ and its decomposition as in Equation (\ref{eq:decom}), that is $\omega=\alpha \wedge \frac{dz}{z}+\beta$, we define a smooth form in the tubular neighborhood by $\overline{\omega}:=\alpha \wedge dz+\beta$. We say that in the tubular neighborhood two $b$-forms $\omega_1, \omega_2$ are $C^\infty$-close if $\overline{\omega_1}, \overline{\omega_2}$ are $C^\infty$-close in the sense of smooth differential forms defined on this neighborhood.  
\end{remark}

We proceed with the proof of Proposition \ref{prop:genericSPOonZ}.

\begin{proof}
Let $R_{\alpha}$ be the Reeb $b$-vector field associated with $\alpha$. It was shown in \cite[Theorem 5.7]{MO18} that $R_{\alpha}|_{Z_k}$ is a (nontrivial) Hamiltonian vector field for each connected component $Z_k$ of $Z$. More precisely, there are functions $H_k\in C^\infty(Z_k,\mathbb R)$ such that $R_\alpha|_{Z_k}=X_{H_k}$, where $X_{H_k}$ is the Hamiltonian field with Hamiltonian function $H_k$ and with respect to a symplectic form that the $b$-contact form induces on the critical set (see \cite[Proposition 6.1]{MO21} for the details of this symplectic form).

Next, we invoke \cite[Theorem 3.1]{FMOP23}, which shows that for generic $\alpha$, each Hamiltonian function $H_k$ is a Morse function. Accordingly, since the first Betti number of $Z_0$ is positive by assumption, the Hamiltonian field $X_{H_0}$ has an even (and positive) number of saddle points $\{p_1,\cdots,p_{2J}\}$, for some $J\geq1$. The vector field $X_{H_0}$ has four semi-orbits (two stable and two unstable) associated with each critical point $p_j$. Since the vector field is Hamiltonian, it is standard that each semi-orbit is part of a homoclinic orbit (i.e., the closure of the orbit contains only one critical point) or of a heteroclinic orbit (the closure of the orbit contains two different critical points). Both cases yield singular periodic orbits on $Z_0$ by definition. This completes the proof of the proposition.     
\end{proof}
\begin{remark}
In fact, an easy counting shows that $R_\alpha$ has $2b_1(Z_0)$ singular periodic orbits on $Z_0$ (in the generic case that the exceptional Hamiltonian of $R_\alpha$ on $Z_0$ is Morse). Here $b_1(Z_0)$ denotes the first Betti number of $Z_0$. Moreover, since a generic Morse function has different critical values for different critical points, these singular periodic orbits are in fact homoclinic orbits of the field.
\end{remark}

\section{A counterexample to the singular Hamiltonian Seifert conjecture}\label{sec:bSeifert}

Following the ideas of the counterexample to the singular Weinstein conjecture, in this section we will tackle the singular Hamiltonian Seifert conjecture. Specifically, we prove the following:

\begin{theorem}\label{thm:R4counterexampleSeifert}
    There exists a Hamiltonian $H\in C^\infty(\mathbb{R}^{4})$ in the standard $b$-symplectic manifold $(\mathbb{R}^{4}, \frac{dx_1}{x_1}\wedge dy_1 +  dx_2\wedge dy_2)$ such that $X_H$ does not have any singular periodic orbits, nor periodic orbits away from the critical surfaces in any of the level-sets $\Sigma_c=H^{-1}(c)$  for $c\in (1-\varepsilon,1+\varepsilon)$. Furthermore, these level-sets are all diffeomorphic to $\mathbb{S}^3$. Here, $\varepsilon>0$ is a small enough constant.
\end{theorem}

This contrasts strongly the Hamiltonian dynamics in smooth symplectic manifolds: in fact, it is known that the set of level-sets with periodic orbits is dense in $\mathbb{R}$, see for instance \cite{HoferZehnder}.

\begin{proof}
    The counterexample is based on the counterexample for the singular Weinstein conjecture of Section \ref{subsec:counterexample}. Consider the double harmonic oscillator Hamiltonian $H\in C^\infty(\mathbb{R}^4)$ given by $H(x_1,y_1,x_2,y_2)= \frac{1}{2}(|x|^2 + |y|^2)$ in the standard $b$-symplectic $\mathbb R^4$. As explained in Section \ref{sec:bhopf}, the Hamiltonian vector field $X_H$ on each level-set $\Sigma_{c}:= H^{-1}(c)$ induces the $b$-Hopf flow, rescaled according to the energy. In particular, on each of these level-sets, there are exactly two singular periodic orbits that pass through the points $(\pm c^{\frac{1}{2}}, 0, 0, 0)$, and no other escape orbits. Thus, if we break the singular periodic orbits with local perturbations, we will achieve the desired example, just as in Section \ref{subsec:counterexample}.

    We now introduce a perturbation of the Hamiltonian in a neighborhood of the points $(\pm 1,0,0,0)$ to break the singular periodic orbits. Since in a small neighborhood of a regular point $p$, a Hamiltonian system can be normalized to the linear Hamiltonian system given by $G=x_1$ in a neighborhood of the origin of the standard symplectic $\mathbb{R}^{4}$, the following lemma suffices to provide the desired perturbation. Again, we remark that this lemma belongs to smooth symplectic geometry, and not to $b$-symplectic geometry.

\begin{lemma}
    Let $(\mathbb{R}^{4},\omega)$ be the (smooth) standard symplectic $\mathbb{R}^{4}$. There exists a Hamiltonian $\tilde G$ arbitrarily $C^\infty$-close to the Hamiltonian $G=x_1$ and differing only in a small neighborhood $U$ of the origin, such that for all energies $c$ in a small interval $I^*$ containing 0, and for all sufficiently small $(x_2,y_2)\neq (0,0)$, the orbits of $\tilde G$ passing through $(c, y_1,x_2,y_2)$ and $(c,-y_1,x_2,y_2)$ do not coincide, where $(c,\pm y_1,x_2,y_2)$ are outside of $U$.
\end{lemma}

In other words, the perturbation applied to the Hamiltonian $G=x_1$ breaks the orbits passing through $(c,0,x_2,y_2)$ with $(x_2,y_2)\neq (0,0)$ (here, by ``break" we mean that the semi-orbits that coincided at $(c,0,x_2,y_2)$ before the perturbation no longer coincide after the perturbation). In particular, if we apply this perturbation to the double harmonic oscillator Hamiltonian in neighborhoods of the points $p_\pm=(\pm 1,0,0,0)$ such that $p_\pm$ are mapped to $(0,0,x_2,y_2)$ with small $(x_2,y_2)\neq(0,0)$, the singular periodic orbits will be successfully broken for all energies in an open set, yielding the desired counterexample. \end{proof}

\begin{proof}[Proof of the lemma]
    
    The construction of the perturbed Hamiltonian $\tilde G$ is similar to the construction in the proof of Proposition \ref{prop:breakspo}. The dynamics of $\tilde G$ on each energy level are essentially the same dynamics as the Reeb dynamics of the contact form in that proposition:  The flowlines are parallel everywhere except in a small cylindrical neighborhood of the origin, in which there is a small twist that breaks orbits entering and exiting the lids of the cylinder (see Figure \ref{fig:breakingspo}). 
    
    Similar to Proposition \ref{prop:breakspo}, we introduce a compact set $K=[-\delta,\delta]^2\times D_\delta$ in which the perturbation is contained. We parametrize the $(x_2,y_2)$-plane with polar coordinates $(r,\varphi)$, so the closed 2-disk $D_\delta$ is given by $r\leq\delta$. Let $\eta = \frac{1}{2}r^2$, so that the corresponding Hamiltonian vector field $X_{\eta} = \partial_\varphi$ is a rotation in the $(x_2,y_2)$-plane. Now consider
    \begin{equation*}
        \tilde{G} = G + \varepsilon f(x_1)f(y_1)f(r)\eta,
    \end{equation*}
   where $f:\mathbb{R}^{}\rightarrow [0,1]$ is a bump function which is equal to $1$ in $I_{ct}:=[-\delta/2,\delta/2]$, and with support in $(-\delta,\delta)$. The Hamiltonian $\tilde G$ can be made arbitrarily $C^\infty$-close to $G$ by making $\varepsilon$ sufficiently small, and they differ only in a compact neighborhood of the origin. Since $\tilde G$ and $G$ are close, there is a subinterval $I^*\subset I_{ct}$ such that for all energies $c\in I^*$, the energy level $\tilde{G}^{-1}(c)$ is contained in $\{x_1\in I_{ct}\}$. We now show that all these energy levels satisfy the statement of the lemma. That is, that for all $c\in I^*$ and all $(x_2,y_2)\in D_{\delta/2}\setminus \{0\}$, the orbits of $\tilde G$ passing through $(c,y_1,x_2,y_2)$ and $(c,-y_1,x_2,y_2)$ do not coincide, where $|y_1|>\delta$. To do this, we integrate the flow explicitly for these points.
   
   Since for $c\in I^*$, we have $f(x_1) = 1$, we can disregard the $f(x_1)$ factor. Furthermore, we consider only flowlines passing through $r\in I_{ct}$, and since $r$ is an invariant of the flow, we also have $f(r) = 1$. Therefore, the Hamiltonian is effectively
  \begin{equation*}
  \tilde G = G + \varepsilon f(y_1)\eta.
  \end{equation*}
The Hamiltonian vector field is easily computed to be
\begin{equation*}
X_{\tilde G} = \partial_{y_1} + \varepsilon\left(f'(y_1)\eta\partial_{x_1} + f(y_1)\partial_\varphi\right).
\end{equation*}
Note, again, that $\eta = \frac{1}{2}r^2$ is invariant along flowlines. Integrating the orbits of this field is much easier than integrating the Reeb vector field of Proposition \ref{prop:breakspo}: Taking as initial condition the point $(c^*,-y_1^*,r^*,\varphi^*)$ with $c^*\in I^*$, $y_1^*>\delta$ and $r^*\in I_{ct}$, we obtain
\begin{equation*}
    c(t) = c^* + \varepsilon\eta\int_{-y_1^*}^{-y_1^*+t}f'(s)ds\;,\; y_1(t) = -y_1^* + t\;,\; r(t) = r^*\;\text{and}\; \varphi(t) = \varphi^* + \varepsilon\int_{-y_1^*}^{-y_1^* + t}f(s)ds.
\end{equation*}
The time for the orbit to reach $\{y_1 = y_1^*\}$ is $T = 2y_1^*$, so keeping in mind that $f$ non-negative and zero at $\pm y_1^*$, we see that
\begin{equation*}
    c(T) = c^*\quad \text{ and } \quad \varphi(T) = \varphi^* + \varepsilon C
\end{equation*}
for some $C>0$. In other words, for small enough $\varepsilon$, the orbit coming into the cylinder at $(c^*,-y_1^*,r^*,\varphi^*)$ exits at $(c^*,y_1^*,r^*,\varphi^* +\varepsilon C)\neq (c^*,y_1^*,r^*,\varphi^*)$. This completes the proof of the lemma.

\end{proof}

\section{A generalized Weinstein conjecture}

We end this article with a reformulation of the singular Weinstein conjecture which remains unsolved. Observe that in the counterexample we provided there are no singular periodic orbits and no regular periodic orbits outside the critical set $Z$.  However, the example exhibits an infinite number of generalized singular periodic orbits.  Indeed, by our procedure, the singular periodic orbits are broken and the resulting $b$-contact manifold has only generalized singular periodic orbits outside of the critical set. It remains uncertain whether constructions without generalized singular periodic orbits can be achieved.

The next definition is needed for the formulation of a  generalized version of the singular Weinstein conjecture.

\begin{definition}
    An orbit in $M \setminus Z$ is called \emph{quasi-closed} if it is either closed 
or its $\alpha$ - and $\omega$-limit sets belong to $Z$.
\end{definition}
In particular, both, regular periodic orbits and singular periodic orbits are quasi-closed. Generalized singular periodic orbits are also quasi-closed.

\begin{center}
    \begin{tikzpicture}
  \draw[thick] (0,0) arc (180:0:2);
  
  \fill[red] (0,0) circle (2pt) node[below] {A};
  \fill[red] (4,0) circle (2pt) node[below] {B};

  \draw[thick][->] (5,0) -- node[below] {$\pi$} (6.5,0);
  \draw[thick, blue] (9,0) circle (2);
    \fill[red] (7,0) circle (2pt) node[right] {$A\sim B$};
 
\end{tikzpicture}
\end{center}
A pictorial way to represent quasi-closed orbits is to think that they are the closed orbits in  $M/Z$,
where the quotient is taken in the topological sense.
This leads us to the formulation of a new conjecture:

\begin{conjecture}[Generalized Weinstein conjecture]\label{conj:gwc}
Every Reeb vector field of a (eventually singular) contact manifold admits at least
one quasi-closed orbit.
\end{conjecture}

The facts that the volume-preserving Seifert conjecture is unsolved in the $C^\infty$ class (see~\cite{Kup}) and there is no topological classification of limit sets for $3$-dimensional flows~\cite{Gu} suggest that this conjecture poses a significant challenge. Therefore, it should be the focus of future research.



\end{document}

%% file: Fig_DependenciesSingularOrbits.tex
\begin{tikzcd}
                                            & \text{Generalized Escape Orbit}                                                &                                                                    \\
                                            &                                                                                &                                                                    \\
\text{Escape Orbit} \arrow[ruu, Rightarrow] &                                                                                & \text{Generalized Singular Periodic Orbit} \arrow[luu, Rightarrow] \\
                                            &                                                                                &                                                                    \\
                                            & \text{Singular Periodic Orbit} \arrow[ruu, Rightarrow] \arrow[luu, Rightarrow] &                                                                   
\end{tikzcd}

%% file: Fig_SPO.tex
\tikzset{every picture/.style={line width=0.75pt}} 

\begin{tikzpicture}[x=0.75pt,y=0.75pt,yscale=-1,xscale=1]

\draw [color={rgb, 255:red, 4; green, 4; blue, 4 }  ,draw opacity=1 ]   (148.75,197.5) .. controls (144.47,126.09) and (163.03,149.08) .. (148.44,109.12) ;
\draw [shift={(147.75,107.25)}, rotate = 69.57] [color={rgb, 255:red, 4; green, 4; blue, 4 }  ,draw opacity=1 ][line width=0.75]    (10.93,-3.29) .. controls (6.95,-1.4) and (3.31,-0.3) .. (0,0) .. controls (3.31,0.3) and (6.95,1.4) .. (10.93,3.29)   ;
\draw  [color={rgb, 255:red, 255; green, 0; blue, 0 }  ,draw opacity=1 ] (100,153.25) .. controls (100,127.71) and (121.71,107) .. (148.5,107) .. controls (175.29,107) and (197,127.71) .. (197,153.25) .. controls (197,178.79) and (175.29,199.5) .. (148.5,199.5) .. controls (121.71,199.5) and (100,178.79) .. (100,153.25) -- cycle ;
\draw [color={rgb, 255:red, 255; green, 0; blue, 0 }  ,draw opacity=1 ] [dash pattern={on 4.5pt off 4.5pt}]  (100,153.25) .. controls (118,163.5) and (164,170.5) .. (197,153.25) ;
\draw  [fill={rgb, 255:red, 3; green, 0; blue, 0 }  ,fill opacity=1 ] (146,107.25) .. controls (146,106.28) and (146.78,105.5) .. (147.75,105.5) .. controls (148.72,105.5) and (149.5,106.28) .. (149.5,107.25) .. controls (149.5,108.22) and (148.72,109) .. (147.75,109) .. controls (146.78,109) and (146,108.22) .. (146,107.25) -- cycle ;
\draw  [fill={rgb, 255:red, 3; green, 0; blue, 0 }  ,fill opacity=1 ] (147,199.25) .. controls (147,198.28) and (147.78,197.5) .. (148.75,197.5) .. controls (149.72,197.5) and (150.5,198.28) .. (150.5,199.25) .. controls (150.5,200.22) and (149.72,201) .. (148.75,201) .. controls (147.78,201) and (147,200.22) .. (147,199.25) -- cycle ;
\draw    (156,88.6) .. controls (147.26,90.88) and (147.59,92.17) .. (147.73,103.61) ;
\draw [shift={(147.75,105.5)}, rotate = 272.68] [color={rgb, 255:red, 0; green, 0; blue, 0 }  ][line width=0.75]    (10.93,-3.29) .. controls (6.95,-1.4) and (3.31,-0.3) .. (0,0) .. controls (3.31,0.3) and (6.95,1.4) .. (10.93,3.29)   ;
\draw  [dash pattern={on 4.5pt off 4.5pt}]  (156,88.6) -- (187.2,80.6) ;

\end{tikzpicture}

%% file: Fig_GardenOfOrbits.tex
\tikzset{every picture/.style={line width=0.75pt}} 

\begin{tikzpicture}[x=0.75pt,y=0.75pt,yscale=-1.5,xscale=1.5]

\draw  [color={rgb, 255:red, 255; green, 0; blue, 0 }  ,draw opacity=1 ] (231.91,210.43) .. controls (208.76,210.61) and (189.8,192.51) .. (189.57,170.01) .. controls (189.34,147.51) and (207.93,129.11) .. (231.09,128.93) .. controls (254.25,128.75) and (273.21,146.85) .. (273.43,169.35) .. controls (273.66,191.86) and (255.07,210.25) .. (231.91,210.43) -- cycle ;
\draw [color={rgb, 255:red, 255; green, 0; blue, 0 }  ,draw opacity=1 ] [dash pattern={on 4.5pt off 4.5pt}]  (231.91,210.43) .. controls (237,190) and (236,153) .. (231.09,128.93) ;
\draw  [fill={rgb, 255:red, 3; green, 0; blue, 0 }  ,fill opacity=1 ] (189.82,172.11) .. controls (188.94,172.12) and (188.23,171.46) .. (188.22,170.65) .. controls (188.21,169.84) and (188.91,169.18) .. (189.79,169.17) .. controls (190.67,169.16) and (191.38,169.81) .. (191.39,170.63) .. controls (191.4,171.44) and (190.7,172.1) .. (189.82,172.11) -- cycle ;
\draw    (172.83,163.84) .. controls (174.78,170.51) and (179.47,170.55) .. (183.81,170.7) ;
\draw [shift={(185.8,170.8)}, rotate = 185.38] [color={rgb, 255:red, 0; green, 0; blue, 0 }  ][line width=0.75]    (6.56,-1.97) .. controls (4.17,-0.84) and (1.99,-0.18) .. (0,0) .. controls (1.99,0.18) and (4.17,0.84) .. (6.56,1.97)   ;
\draw  [dash pattern={on 4.5pt off 4.5pt}]  (172.83,163.84) -- (165.31,137.68) ;
\draw  [color={rgb, 255:red, 255; green, 0; blue, 0 }  ,draw opacity=1 ] (372.32,208.99) .. controls (349.16,209.17) and (330.2,191.08) .. (329.97,168.57) .. controls (329.75,146.07) and (348.34,127.68) .. (371.49,127.49) .. controls (394.65,127.31) and (413.61,145.41) .. (413.84,167.91) .. controls (414.07,190.42) and (395.48,208.81) .. (372.32,208.99) -- cycle ;
\draw [color={rgb, 255:red, 255; green, 0; blue, 0 }  ,draw opacity=1 ] [dash pattern={on 4.5pt off 4.5pt}]  (372.32,208.99) .. controls (377,193) and (378,158) .. (371.5,127.5) ;
\draw  [color={rgb, 255:red, 255; green, 0; blue, 0 }  ,draw opacity=1 ] (393.89,203.66) .. controls (391.72,203.67) and (389.81,188.02) .. (389.61,168.69) .. controls (389.41,149.36) and (391.01,133.68) .. (393.18,133.66) .. controls (395.35,133.64) and (397.27,149.3) .. (397.47,168.63) .. controls (397.66,187.96) and (396.06,203.64) .. (393.89,203.66) -- cycle ;
\draw  [color={rgb, 255:red, 255; green, 0; blue, 0 }  ,draw opacity=1 ] (395.8,154.54) -- (396.52,147.69) -- (398.28,154.02) ;
\draw  [color={rgb, 255:red, 255; green, 0; blue, 0 }  ,draw opacity=1 ] (388.51,178.88) -- (390.34,185.14) -- (390.98,178.26) ;

\draw    (435.06,139.03) .. controls (429.95,117.93) and (432,216) .. (426,198) .. controls (420.12,180.36) and (413.28,109.71) .. (406.04,134.8) ;
\draw [shift={(405.59,136.46)}, rotate = 284.01] [color={rgb, 255:red, 0; green, 0; blue, 0 }  ][line width=0.75]    (4.37,-1.96) .. controls (2.78,-0.92) and (1.32,-0.27) .. (0,0) .. controls (1.32,0.27) and (2.78,0.92) .. (4.37,1.96)   ;
\draw [shift={(430.73,170.3)}, rotate = 272.6] [color={rgb, 255:red, 0; green, 0; blue, 0 }  ][line width=0.75]    (4.37,-1.96) .. controls (2.78,-0.92) and (1.32,-0.27) .. (0,0) .. controls (1.32,0.27) and (2.78,0.92) .. (4.37,1.96)   ;
\draw [shift={(417.83,157.13)}, rotate = 79.53] [color={rgb, 255:red, 0; green, 0; blue, 0 }  ][line width=0.75]    (4.37,-1.96) .. controls (2.78,-0.92) and (1.32,-0.27) .. (0,0) .. controls (1.32,0.27) and (2.78,0.92) .. (4.37,1.96)   ;
\draw  [dash pattern={on 4.5pt off 4.5pt}]  (447,149.24) .. controls (439.55,252.3) and (439.86,170.79) .. (435.06,139.03) ;
\draw  [color={rgb, 255:red, 255; green, 0; blue, 0 }  ,draw opacity=1 ] (252.92,205.26) .. controls (250.75,205.28) and (248.83,189.4) .. (248.64,169.8) .. controls (248.44,150.19) and (250.04,134.28) .. (252.21,134.27) .. controls (254.38,134.25) and (256.3,150.13) .. (256.49,169.73) .. controls (256.69,189.34) and (255.09,205.25) .. (252.92,205.26) -- cycle ;
\draw  [color={rgb, 255:red, 255; green, 0; blue, 0 }  ,draw opacity=1 ] (254.82,155.44) -- (255.55,148.5) -- (257.3,154.92) ;
\draw  [color={rgb, 255:red, 255; green, 0; blue, 0 }  ,draw opacity=1 ] (247.54,180.13) -- (249.37,186.49) -- (250.01,179.51) ;

\draw    (300,137) .. controls (298,148) and (290,212) .. (285,201) .. controls (280.1,190.22) and (272.33,110.68) .. (265.06,135.42) ;
\draw [shift={(264.62,137.07)}, rotate = 284.01] [color={rgb, 255:red, 0; green, 0; blue, 0 }  ][line width=0.75]    (4.37,-1.96) .. controls (2.78,-0.92) and (1.32,-0.27) .. (0,0) .. controls (1.32,0.27) and (2.78,0.92) .. (4.37,1.96)   ;
\draw [shift={(294.22,173.86)}, rotate = 279.75] [color={rgb, 255:red, 0; green, 0; blue, 0 }  ][line width=0.75]    (4.37,-1.96) .. controls (2.78,-0.92) and (1.32,-0.27) .. (0,0) .. controls (1.32,0.27) and (2.78,0.92) .. (4.37,1.96)   ;
\draw [shift={(276.85,159.88)}, rotate = 79.73] [color={rgb, 255:red, 0; green, 0; blue, 0 }  ][line width=0.75]    (4.37,-1.96) .. controls (2.78,-0.92) and (1.32,-0.27) .. (0,0) .. controls (1.32,0.27) and (2.78,0.92) .. (4.37,1.96)   ;
\draw  [color={rgb, 255:red, 255; green, 0; blue, 0 }  ,draw opacity=1 ] (350.29,203.5) .. controls (348.12,203.52) and (346.21,187.86) .. (346.01,168.53) .. controls (345.81,149.2) and (347.42,133.52) .. (349.59,133.5) .. controls (351.76,133.48) and (353.67,149.14) .. (353.87,168.47) .. controls (354.06,187.8) and (352.46,203.48) .. (350.29,203.5) -- cycle ;
\draw  [color={rgb, 255:red, 255; green, 0; blue, 0 }  ,draw opacity=1 ] (352.2,154.38) -- (352.93,147.53) -- (354.68,153.86) ;
\draw  [color={rgb, 255:red, 255; green, 0; blue, 0 }  ,draw opacity=1 ] (344.91,178.72) -- (346.74,184.99) -- (347.38,178.11) ;

\draw    (343,134) .. controls (337.8,123.54) and (341.17,147.61) .. (340,163) .. controls (338.83,178.39) and (333,197) .. (328,204) .. controls (323.08,210.9) and (306.51,97.48) .. (300.28,135.17) ;
\draw [shift={(300,137)}, rotate = 277.94] [color={rgb, 255:red, 0; green, 0; blue, 0 }  ][line width=0.75]    (4.37,-1.96) .. controls (2.78,-0.92) and (1.32,-0.27) .. (0,0) .. controls (1.32,0.27) and (2.78,0.92) .. (4.37,1.96)   ;
\draw [shift={(340.18,147.97)}, rotate = 268.67] [color={rgb, 255:red, 0; green, 0; blue, 0 }  ][line width=0.75]    (4.37,-1.96) .. controls (2.78,-0.92) and (1.32,-0.27) .. (0,0) .. controls (1.32,0.27) and (2.78,0.92) .. (4.37,1.96)   ;
\draw [shift={(335.44,186.65)}, rotate = 286.38] [color={rgb, 255:red, 0; green, 0; blue, 0 }  ][line width=0.75]    (4.37,-1.96) .. controls (2.78,-0.92) and (1.32,-0.27) .. (0,0) .. controls (1.32,0.27) and (2.78,0.92) .. (4.37,1.96)   ;
\draw [shift={(314.79,158.72)}, rotate = 76.44] [color={rgb, 255:red, 0; green, 0; blue, 0 }  ][line width=0.75]    (4.37,-1.96) .. controls (2.78,-0.92) and (1.32,-0.27) .. (0,0) .. controls (1.32,0.27) and (2.78,0.92) .. (4.37,1.96)   ;
\draw [color={rgb, 255:red, 4; green, 4; blue, 4 }  ,draw opacity=1 ]   (271.62,169.32) .. controls (206.92,173.43) and (231.09,157.28) .. (195.09,169.81) ;
\draw [shift={(193.4,170.4)}, rotate = 340.49] [color={rgb, 255:red, 4; green, 4; blue, 4 }  ,draw opacity=1 ][line width=0.75]    (6.56,-1.97) .. controls (4.17,-0.84) and (1.99,-0.18) .. (0,0) .. controls (1.99,0.18) and (4.17,0.84) .. (6.56,1.97)   ;
\draw  [fill={rgb, 255:red, 3; green, 0; blue, 0 }  ,fill opacity=1 ] (273.22,170.78) .. controls (272.35,170.78) and (271.63,170.13) .. (271.62,169.32) .. controls (271.61,168.51) and (272.32,167.84) .. (273.19,167.84) .. controls (274.07,167.83) and (274.79,168.48) .. (274.8,169.29) .. controls (274.8,170.11) and (274.1,170.77) .. (273.22,170.78) -- cycle ;

\end{tikzpicture}

%% file: Fig_bReeb.tex
\tikzset{every picture/.style={line width=0.75pt}} 

\begin{tikzpicture}[x=0.75pt,y=0.75pt,yscale=-1.3,xscale=1.3]

\draw   (302.63,126.58) .. controls (302.63,102.7) and (324.2,83.34) .. (350.8,83.34) .. controls (377.4,83.34) and (398.97,102.7) .. (398.97,126.58) .. controls (398.97,150.45) and (377.4,169.81) .. (350.8,169.81) .. controls (324.2,169.81) and (302.63,150.45) .. (302.63,126.58) -- cycle ;
\draw    (230.33,144.17) -- (415.06,144.06) ;
\draw    (282.24,110.27) -- (306,110.17) ;
\draw    (396,111.17) -- (466.97,110.17) ;
\draw  [dash pattern={on 4.5pt off 4.5pt}] (302.63,126.58) .. controls (302.63,123.3) and (324.2,120.65) .. (350.8,120.65) .. controls (377.4,120.65) and (398.97,123.3) .. (398.97,126.58) .. controls (398.97,129.85) and (377.4,132.5) .. (350.8,132.5) .. controls (324.2,132.5) and (302.63,129.85) .. (302.63,126.58) -- cycle ;
\draw    (415.06,144.06) -- (466.97,110.17) ;
\draw    (230.33,144.17) -- (282.24,110.27) ;
\draw  [fill={rgb, 255:red, 0; green, 0; blue, 0 }  ,fill opacity=1 ] (348.48,83.34) .. controls (348.48,82.77) and (349,82.3) .. (349.64,82.3) .. controls (350.28,82.3) and (350.8,82.77) .. (350.8,83.34) .. controls (350.8,83.91) and (350.28,84.38) .. (349.64,84.38) .. controls (349,84.38) and (348.48,83.91) .. (348.48,83.34) -- cycle ;
\draw  [fill={rgb, 255:red, 0; green, 0; blue, 0 }  ,fill opacity=1 ] (349.81,169.96) .. controls (349.81,169.39) and (350.33,168.92) .. (350.97,168.92) .. controls (351.61,168.92) and (352.13,169.39) .. (352.13,169.96) .. controls (352.13,170.53) and (351.61,171) .. (350.97,171) .. controls (350.33,171) and (349.81,170.53) .. (349.81,169.96) -- cycle ;
\draw   (315,98.4) .. controls (315,96.13) and (331.04,94.29) .. (350.83,94.29) .. controls (370.62,94.29) and (386.67,96.13) .. (386.67,98.4) .. controls (386.67,100.66) and (370.62,102.5) .. (350.83,102.5) .. controls (331.04,102.5) and (315,100.66) .. (315,98.4) -- cycle ;
\draw   (365.27,100.9) -- (372.28,101.73) -- (365.78,103.5) ;
\draw   (340.41,93.03) -- (333.98,94.88) -- (341.02,95.63) ;

\draw    (351,180.9) -- (351,174.83) ;
\draw [shift={(351,172.83)}, rotate = 90] [color={rgb, 255:red, 0; green, 0; blue, 0 }  ][line width=0.75]    (4.37,-1.32) .. controls (2.78,-0.56) and (1.32,-0.12) .. (0,0) .. controls (1.32,0.12) and (2.78,0.56) .. (4.37,1.32)   ;
\draw    (349.73,79.33) -- (349.09,32.67) ;
\draw [shift={(349.07,30.67)}, rotate = 89.22] [color={rgb, 255:red, 0; green, 0; blue, 0 }  ][line width=0.75]    (4.37,-1.32) .. controls (2.78,-0.56) and (1.32,-0.12) .. (0,0) .. controls (1.32,0.12) and (2.78,0.56) .. (4.37,1.32)   ;
\draw   (313.67,153.37) .. controls (313.67,151.45) and (330.31,149.9) .. (350.83,149.9) .. controls (371.36,149.9) and (388,151.45) .. (388,153.37) .. controls (388,155.28) and (371.36,156.84) .. (350.83,156.84) .. controls (330.31,156.84) and (313.67,155.28) .. (313.67,153.37) -- cycle ;
\draw   (365.81,155.49) -- (373.08,156.19) -- (366.34,157.69) ;
\draw   (340.02,148.83) -- (333.35,150.4) -- (340.65,151.03) ;

\draw    (318.2,225) .. controls (310.87,217.67) and (333.43,207.38) .. (354.57,203.05) .. controls (375.72,198.72) and (386.13,195.13) .. (379.8,189.8) .. controls (373.47,184.47) and (304.6,176.6) .. (320.2,175.4) .. controls (335.49,174.22) and (409.17,170.74) .. (384.25,164.2) ;
\draw [shift={(382.6,163.8)}, rotate = 12.94] [color={rgb, 255:red, 0; green, 0; blue, 0 }  ][line width=0.75]    (4.37,-1.96) .. controls (2.78,-0.92) and (1.32,-0.27) .. (0,0) .. controls (1.32,0.27) and (2.78,0.92) .. (4.37,1.96)   ;
\draw [shift={(334.3,208.96)}, rotate = 158.41] [color={rgb, 255:red, 0; green, 0; blue, 0 }  ][line width=0.75]    (4.37,-1.96) .. controls (2.78,-0.92) and (1.32,-0.27) .. (0,0) .. controls (1.32,0.27) and (2.78,0.92) .. (4.37,1.96)   ;
\draw [shift={(373.15,198.49)}, rotate = 162.13] [color={rgb, 255:red, 0; green, 0; blue, 0 }  ][line width=0.75]    (4.37,-1.96) .. controls (2.78,-0.92) and (1.32,-0.27) .. (0,0) .. controls (1.32,0.27) and (2.78,0.92) .. (4.37,1.96)   ;
\draw [shift={(345.18,181.64)}, rotate = 9.9] [color={rgb, 255:red, 0; green, 0; blue, 0 }  ][line width=0.75]    (4.37,-1.96) .. controls (2.78,-0.92) and (1.32,-0.27) .. (0,0) .. controls (1.32,0.27) and (2.78,0.92) .. (4.37,1.96)   ;
\draw [shift={(360.92,172.44)}, rotate = 174.78] [color={rgb, 255:red, 0; green, 0; blue, 0 }  ][line width=0.75]    (4.37,-1.96) .. controls (2.78,-0.92) and (1.32,-0.27) .. (0,0) .. controls (1.32,0.27) and (2.78,0.92) .. (4.37,1.96)   ;
\draw    (351.13,184.73) -- (351.22,219.22) ;
\draw [color={rgb, 255:red, 255; green, 0; blue, 0 }  ,draw opacity=1 ]   (349.71,86.56) -- (350.16,101) ;
\draw [color={rgb, 255:red, 255; green, 0; blue, 0 }  ,draw opacity=1 ]   (350.16,104.78) -- (350.6,154.78) ;
\draw [color={rgb, 255:red, 255; green, 0; blue, 0 }  ,draw opacity=1 ]   (351,158.56) -- (351,165.35) ;
\draw [shift={(351,167.35)}, rotate = 270] [color={rgb, 255:red, 255; green, 0; blue, 0 }  ,draw opacity=1 ][line width=0.75]    (4.37,-1.32) .. controls (2.78,-0.56) and (1.32,-0.12) .. (0,0) .. controls (1.32,0.12) and (2.78,0.56) .. (4.37,1.32)   ;

\draw (373.2,70.8) node [anchor=north west][inner sep=0.75pt]  [font=\scriptsize] [align=left] {$\displaystyle Z\ =\ \mathbb{S}^{2}$};

\end{tikzpicture}

%% file: Fig_extendwithstandard.tex
\tikzset{every picture/.style={line width=0.75pt}} 

\begin{tikzpicture}[x=0.75pt,y=0.75pt,yscale=-1,xscale=1]

\draw   (67.33,70) -- (289.67,70) -- (289.67,239.67) -- (67.33,239.67) -- cycle ;
\draw   (392.67,70.67) -- (625,70.67) -- (625,239.67) -- (392.67,239.67) -- cycle ;
\draw    (243.99,151.05) -- (445.49,150.06) ;
\draw [shift={(447.49,150.05)}, rotate = 179.72] [color={rgb, 255:red, 0; green, 0; blue, 0 }  ][line width=0.75]    (10.93,-3.29) .. controls (6.95,-1.4) and (3.31,-0.3) .. (0,0) .. controls (3.31,0.3) and (6.95,1.4) .. (10.93,3.29)   ;
\draw  [color={rgb, 255:red, 255; green, 0; blue, 0 }  ,draw opacity=1 ] (143.67,154.83) .. controls (143.67,135.6) and (159.26,120) .. (178.5,120) .. controls (197.74,120) and (213.33,135.6) .. (213.33,154.83) .. controls (213.33,174.07) and (197.74,189.67) .. (178.5,189.67) .. controls (159.26,189.67) and (143.67,174.07) .. (143.67,154.83) -- cycle ;
\draw  [dash pattern={on 4.5pt off 4.5pt}] (122.43,154.83) .. controls (122.43,124.31) and (147.53,99.57) .. (178.5,99.57) .. controls (209.47,99.57) and (234.57,124.31) .. (234.57,154.83) .. controls (234.57,185.36) and (209.47,210.1) .. (178.5,210.1) .. controls (147.53,210.1) and (122.43,185.36) .. (122.43,154.83)(106.99,154.83) .. controls (106.99,115.78) and (139.01,84.13) .. (178.5,84.13) .. controls (217.99,84.13) and (250.01,115.78) .. (250.01,154.83) .. controls (250.01,193.88) and (217.99,225.54) .. (178.5,225.54) .. controls (139.01,225.54) and (106.99,193.88) .. (106.99,154.83) ;
\draw  [color={rgb, 255:red, 255; green, 0; blue, 0 }  ,draw opacity=1 ] (481,154.83) .. controls (481,135.6) and (496.6,120) .. (515.83,120) .. controls (535.07,120) and (550.67,135.6) .. (550.67,154.83) .. controls (550.67,174.07) and (535.07,189.67) .. (515.83,189.67) .. controls (496.6,189.67) and (481,174.07) .. (481,154.83) -- cycle ;
\draw  [dash pattern={on 4.5pt off 4.5pt}] (459.76,154.83) .. controls (459.76,124.31) and (484.87,99.57) .. (515.83,99.57) .. controls (546.8,99.57) and (571.9,124.31) .. (571.9,154.83) .. controls (571.9,185.36) and (546.8,210.1) .. (515.83,210.1) .. controls (484.87,210.1) and (459.76,185.36) .. (459.76,154.83) -- cycle ;
\draw  [dash pattern={on 4.5pt off 4.5pt}]  (451.11,101) .. controls (468.49,80.55) and (529.07,88.48) .. (546.99,95.05) .. controls (564.92,101.62) and (582.66,116.97) .. (597.99,136.16) .. controls (613.32,155.35) and (604.26,188.2) .. (586.49,203.55) .. controls (568.73,218.9) and (540.49,216.05) .. (513.99,225.05) .. controls (487.49,234.05) and (445.82,200.06) .. (441.49,171.05) .. controls (437.17,142.05) and (435.49,117.05) .. (451.11,101) -- cycle ;

\draw (82.67,80.33) node [anchor=north west][inner sep=0.75pt]  [font=\footnotesize] [align=left] {$\displaystyle \xi _{b}$};
\draw (406.67,79.67) node [anchor=north west][inner sep=0.75pt]  [font=\footnotesize] [align=left] {$\displaystyle \xi _{std}$};
\draw (464,143) node [anchor=north west][inner sep=0.75pt]  [font=\footnotesize] [align=left] {$\displaystyle \xi _{b}$};

\end{tikzpicture}

%% file: Fig_extension.tex
\tikzset{every picture/.style={line width=0.75pt}} 

\begin{tikzpicture}[x=0.75pt,y=0.75pt,yscale=-1.3,xscale=1.3]

\draw   (290.5,142.5) .. controls (290.5,133.94) and (297.44,127) .. (306,127) .. controls (314.56,127) and (321.5,133.94) .. (321.5,142.5) .. controls (321.5,151.06) and (314.56,158) .. (306,158) .. controls (297.44,158) and (290.5,151.06) .. (290.5,142.5) -- cycle ;
\draw  [dash pattern={on 4.5pt off 4.5pt}] (282.5,113.25) .. controls (302.5,103.25) and (346.5,104) .. (346.5,124.25) .. controls (346.5,144.5) and (336.5,150.75) .. (331,165.25) .. controls (325.5,179.75) and (319.5,197.5) .. (288.5,179.25) .. controls (257.5,161) and (262.5,123.25) .. (282.5,113.25) -- cycle ;
\draw  [dash pattern={on 4.5pt off 4.5pt}] (256,94.75) .. controls (264.5,85.75) and (345,79.25) .. (362,100.25) .. controls (379,121.25) and (427,97.25) .. (421.5,114.75) .. controls (416,132.25) and (354,196.75) .. (325.5,208.75) .. controls (297,220.75) and (264.92,202.24) .. (250.5,193.75) .. controls (236.08,185.26) and (247.5,103.75) .. (256,94.75) -- cycle ;

\draw (307,136.5) node [anchor=north west][inner sep=0.75pt]  [font=\scriptsize] [align=left] {$\displaystyle Z_{i}$};
\draw (262.5,95.5) node [anchor=north west][inner sep=0.75pt]  [font=\scriptsize] [align=left] {$ S_{i}$};
\draw (374.5,72.5) node [anchor=north west][inner sep=0.75pt]  [font=\scriptsize] [align=left] {$\displaystyle \hat{\alpha }$};
\draw (320.5,121) node [anchor=north west][inner sep=0.75pt]  [font=\scriptsize] [align=left] {$\displaystyle \alpha _{i}$};
\draw (357.5,117) node [anchor=north west][inner sep=0.75pt]  [font=\scriptsize] [align=left] {$\displaystyle \alpha _{i} \ =\ f_{i}\hat{\alpha }$};

\end{tikzpicture}

%% file: Fig_breakingspo.tex
\tikzset{every picture/.style={line width=0.75pt}} 

\begin{tikzpicture}[x=0.75pt,y=0.75pt,yscale=-1,xscale=1]

\draw   (76,44.83) .. controls (76,31.03) and (87.19,19.83) .. (101,19.83) .. controls (114.81,19.83) and (126,31.03) .. (126,44.83) .. controls (126,58.64) and (114.81,69.83) .. (101,69.83) .. controls (87.19,69.83) and (76,58.64) .. (76,44.83) -- cycle ;
\draw   (75,207.83) .. controls (75,194.03) and (86.19,182.83) .. (100,182.83) .. controls (113.81,182.83) and (125,194.03) .. (125,207.83) .. controls (125,221.64) and (113.81,232.83) .. (100,232.83) .. controls (86.19,232.83) and (75,221.64) .. (75,207.83) -- cycle ;
\draw    (100.98,71.83) -- (100,182.83) ;
\draw [shift={(101,69.83)}, rotate = 90.51] [color={rgb, 255:red, 0; green, 0; blue, 0 }  ][line width=0.75]    (10.93,-3.29) .. controls (6.95,-1.4) and (3.31,-0.3) .. (0,0) .. controls (3.31,0.3) and (6.95,1.4) .. (10.93,3.29)   ;
\draw   (304.5,46.33) .. controls (304.5,32.53) and (315.69,21.33) .. (329.5,21.33) .. controls (343.31,21.33) and (354.5,32.53) .. (354.5,46.33) .. controls (354.5,60.14) and (343.31,71.33) .. (329.5,71.33) .. controls (315.69,71.33) and (304.5,60.14) .. (304.5,46.33) -- cycle ;
\draw   (303.5,209.33) .. controls (303.5,195.53) and (314.69,184.33) .. (328.5,184.33) .. controls (342.31,184.33) and (353.5,195.53) .. (353.5,209.33) .. controls (353.5,223.14) and (342.31,234.33) .. (328.5,234.33) .. controls (314.69,234.33) and (303.5,223.14) .. (303.5,209.33) -- cycle ;
\draw    (143.5,122.5) .. controls (145.15,120.81) and (146.81,120.79) .. (148.5,122.43) .. controls (150.19,124.08) and (151.85,124.06) .. (153.5,122.37) .. controls (155.15,120.68) and (156.81,120.66) .. (158.5,122.3) .. controls (160.19,123.95) and (161.85,123.93) .. (163.5,122.24) .. controls (165.15,120.55) and (166.81,120.53) .. (168.5,122.17) .. controls (170.19,123.82) and (171.85,123.8) .. (173.5,122.11) .. controls (175.15,120.42) and (176.81,120.4) .. (178.5,122.04) .. controls (180.19,123.69) and (181.85,123.67) .. (183.5,121.98) .. controls (185.15,120.29) and (186.81,120.27) .. (188.5,121.91) .. controls (190.19,123.56) and (191.85,123.54) .. (193.5,121.85) .. controls (195.15,120.16) and (196.81,120.14) .. (198.5,121.78) .. controls (200.19,123.42) and (201.85,123.4) .. (203.49,121.71) .. controls (205.14,120.02) and (206.8,120) .. (208.49,121.65) .. controls (210.18,123.29) and (211.84,123.27) .. (213.49,121.58) .. controls (215.14,119.89) and (216.8,119.87) .. (218.49,121.52) .. controls (220.18,123.16) and (221.84,123.14) .. (223.49,121.45) .. controls (225.14,119.76) and (226.8,119.74) .. (228.49,121.39) -- (229,121.38) -- (237,121.28) ;
\draw [shift={(239,121.25)}, rotate = 179.25] [color={rgb, 255:red, 0; green, 0; blue, 0 }  ][line width=0.75]    (10.93,-3.29) .. controls (6.95,-1.4) and (3.31,-0.3) .. (0,0) .. controls (3.31,0.3) and (6.95,1.4) .. (10.93,3.29)   ;
\draw   (318.5,91.5) -- (387.5,91.5) -- (387.5,165.25) -- (318.5,165.25) -- cycle ;
\draw    (328.5,184.33) .. controls (331.49,125.05) and (368.63,143.07) .. (373.92,70.35) ;
\draw [shift={(374,69.25)}, rotate = 93.87] [color={rgb, 255:red, 0; green, 0; blue, 0 }  ][line width=0.75]    (10.93,-3.29) .. controls (6.95,-1.4) and (3.31,-0.3) .. (0,0) .. controls (3.31,0.3) and (6.95,1.4) .. (10.93,3.29)   ;
\draw  [dash pattern={on 4.5pt off 4.5pt}]  (395,164) .. controls (430.46,148.98) and (427.11,115.77) .. (395.94,92.79) ;
\draw [shift={(394.5,91.75)}, rotate = 35.29] [color={rgb, 255:red, 0; green, 0; blue, 0 }  ][line width=0.75]    (10.93,-3.29) .. controls (6.95,-1.4) and (3.31,-0.3) .. (0,0) .. controls (3.31,0.3) and (6.95,1.4) .. (10.93,3.29)   ;
\draw  [draw opacity=0][fill={rgb, 255:red, 255; green, 255; blue, 255 }  ,fill opacity=1 ] (347.61,128.38) .. controls (347.61,125.4) and (350.02,122.98) .. (353,122.98) .. controls (355.98,122.98) and (358.39,125.4) .. (358.39,128.38) .. controls (358.39,131.35) and (355.98,133.77) .. (353,133.77) .. controls (350.02,133.77) and (347.61,131.35) .. (347.61,128.38) -- cycle ;
\draw    (329.43,73.36) .. controls (327.47,139.2) and (375,112.47) .. (375,183.25) ;
\draw [shift={(329.5,71.33)}, rotate = 92.51] [color={rgb, 255:red, 0; green, 0; blue, 0 }  ][line width=0.75]    (10.93,-3.29) .. controls (6.95,-1.4) and (3.31,-0.3) .. (0,0) .. controls (3.31,0.3) and (6.95,1.4) .. (10.93,3.29)   ;

\draw (254.5,133.5) node [anchor=north west][inner sep=0.75pt]  [font=\tiny] [align=left] {$\displaystyle \mathbb{D}^{2} \times [ 0,1]$};
\draw (423.5,124) node [anchor=north west][inner sep=0.75pt]  [font=\footnotesize] [align=left] {rotation};
\draw (257,101) node [anchor=north west][inner sep=0.75pt]  [font=\tiny] [align=left] {$\displaystyle \{r\ \in I_{ct}\}$};
\draw (284.54,116.02) node [anchor=north west][inner sep=0.75pt]  [font=\scriptsize,rotate=-90.31] [align=left] {$\displaystyle \cong $};

\end{tikzpicture}

%% file: Fig_bHopf3D.tex
\tikzset{every picture/.style={line width=0.75pt}} 

\begin{tikzpicture}[x=0.75pt,y=0.75pt,yscale=-1.2,xscale=1.2]

\draw   (281.63,144.58) .. controls (281.63,120.7) and (303.2,101.34) .. (329.8,101.34) .. controls (356.4,101.34) and (377.97,120.7) .. (377.97,144.58) .. controls (377.97,168.45) and (356.4,187.81) .. (329.8,187.81) .. controls (303.2,187.81) and (281.63,168.45) .. (281.63,144.58) -- cycle ;
\draw  [fill={rgb, 255:red, 0; green, 0; blue, 0 }  ,fill opacity=1 ] (327.48,101.34) .. controls (327.48,100.77) and (328,100.3) .. (328.64,100.3) .. controls (329.28,100.3) and (329.8,100.77) .. (329.8,101.34) .. controls (329.8,101.91) and (329.28,102.38) .. (328.64,102.38) .. controls (328,102.38) and (327.48,101.91) .. (327.48,101.34) -- cycle ;
\draw  [fill={rgb, 255:red, 0; green, 0; blue, 0 }  ,fill opacity=1 ] (328.81,187.96) .. controls (328.81,187.39) and (329.33,186.92) .. (329.97,186.92) .. controls (330.61,186.92) and (331.13,187.39) .. (331.13,187.96) .. controls (331.13,188.53) and (330.61,189) .. (329.97,189) .. controls (329.33,189) and (328.81,188.53) .. (328.81,187.96) -- cycle ;
\draw   (294,116.4) .. controls (294,114.13) and (310.04,112.29) .. (329.83,112.29) .. controls (349.62,112.29) and (365.67,114.13) .. (365.67,116.4) .. controls (365.67,118.66) and (349.62,120.5) .. (329.83,120.5) .. controls (310.04,120.5) and (294,118.66) .. (294,116.4) -- cycle ;
\draw   (344.27,118.9) -- (351.28,119.73) -- (344.78,121.5) ;
\draw    (330.33,224.67) -- (330.02,192.83) ;
\draw [shift={(330,190.83)}, rotate = 89.44] [color={rgb, 255:red, 0; green, 0; blue, 0 }  ][line width=0.75]    (4.37,-1.32) .. controls (2.78,-0.56) and (1.32,-0.12) .. (0,0) .. controls (1.32,0.12) and (2.78,0.56) .. (4.37,1.32)   ;
\draw    (328.73,97.33) -- (328.99,62.67) ;
\draw [shift={(329,60.67)}, rotate = 90.42] [color={rgb, 255:red, 0; green, 0; blue, 0 }  ][line width=0.75]    (4.37,-1.32) .. controls (2.78,-0.56) and (1.32,-0.12) .. (0,0) .. controls (1.32,0.12) and (2.78,0.56) .. (4.37,1.32)   ;
\draw   (292.67,171.37) .. controls (292.67,169.45) and (309.31,167.9) .. (329.83,167.9) .. controls (350.36,167.9) and (367,169.45) .. (367,171.37) .. controls (367,173.28) and (350.36,174.84) .. (329.83,174.84) .. controls (309.31,174.84) and (292.67,173.28) .. (292.67,171.37) -- cycle ;
\draw   (344.81,173.49) -- (352.08,174.19) -- (345.34,175.69) ;
\draw   (319.02,166.83) -- (312.35,168.4) -- (319.65,169.03) ;

\draw [color={rgb, 255:red, 3; green, 0; blue, 0 }  ,draw opacity=1 ]   (328.71,104.56) -- (329.16,119) ;
\draw [color={rgb, 255:red, 3; green, 0; blue, 0 }  ,draw opacity=1 ]   (329.16,122.78) -- (329.6,172.78) ;
\draw [color={rgb, 255:red, 3; green, 0; blue, 0 }  ,draw opacity=1 ]   (330,176.56) -- (330,183.35) ;
\draw [shift={(330,185.35)}, rotate = 270] [color={rgb, 255:red, 3; green, 0; blue, 0 }  ,draw opacity=1 ][line width=0.75]    (4.37,-1.32) .. controls (2.78,-0.56) and (1.32,-0.12) .. (0,0) .. controls (1.32,0.12) and (2.78,0.56) .. (4.37,1.32)   ;
\draw    (319.33,177.17) .. controls (332.67,177.83) and (371.83,182.7) .. (393.58,173.1) .. controls (415.33,163.5) and (421.39,141.88) .. (417.67,122.83) .. controls (413.95,103.79) and (398.33,95.33) .. (377.33,97.83) .. controls (356.33,100.33) and (334.7,107.16) .. (304.67,108.5) .. controls (275.38,109.8) and (312.71,114.27) .. (337.76,117.27) ;
\draw [shift={(339.67,117.5)}, rotate = 186.84] [color={rgb, 255:red, 0; green, 0; blue, 0 }  ][line width=0.75]    (7.65,-2.3) .. controls (4.86,-0.97) and (2.31,-0.21) .. (0,0) .. controls (2.31,0.21) and (4.86,0.98) .. (7.65,2.3)   ;

\draw (356.7,75.3) node [anchor=north west][inner sep=0.75pt]  [font=\scriptsize] [align=left] {$\displaystyle Z\ =\ \mathbb{S}^{2}$};
\draw (331.6,140) node [anchor=north west][inner sep=0.75pt]  [font=\tiny] [align=left] {$\displaystyle a)$};
\draw (333.13,190.96) node [anchor=north west][inner sep=0.75pt]  [font=\tiny] [align=left] {$\displaystyle b)$};
\draw (295.6,158.8) node [anchor=north west][inner sep=0.75pt]  [font=\tiny] [align=left] {$\displaystyle c)$};
\draw (421.2,118) node [anchor=north west][inner sep=0.75pt]  [font=\tiny] [align=left] {$\displaystyle d)$};

\end{tikzpicture}

%% file: Fig_bHopfPerturbed.tex
\tikzset{every picture/.style={line width=0.75pt}} 

\begin{tikzpicture}[x=0.75pt,y=0.75pt,yscale=-1.2,xscale=1.2]

\draw    (346.57,117.43) .. controls (357.71,117.14) and (344,122.44) .. (337,129.3) .. controls (330,136.16) and (322.2,145.2) .. (323.67,177.5) ;
\draw   (275.63,147.58) .. controls (275.63,123.7) and (297.2,104.34) .. (323.8,104.34) .. controls (350.4,104.34) and (371.97,123.7) .. (371.97,147.58) .. controls (371.97,171.45) and (350.4,190.81) .. (323.8,190.81) .. controls (297.2,190.81) and (275.63,171.45) .. (275.63,147.58) -- cycle ;
\draw  [fill={rgb, 255:red, 0; green, 0; blue, 0 }  ,fill opacity=1 ] (321.48,104.34) .. controls (321.48,103.77) and (322,103.3) .. (322.64,103.3) .. controls (323.28,103.3) and (323.8,103.77) .. (323.8,104.34) .. controls (323.8,104.91) and (323.28,105.38) .. (322.64,105.38) .. controls (322,105.38) and (321.48,104.91) .. (321.48,104.34) -- cycle ;
\draw  [fill={rgb, 255:red, 0; green, 0; blue, 0 }  ,fill opacity=1 ] (322.81,190.96) .. controls (322.81,190.39) and (323.33,189.92) .. (323.97,189.92) .. controls (324.61,189.92) and (325.13,190.39) .. (325.13,190.96) .. controls (325.13,191.53) and (324.61,192) .. (323.97,192) .. controls (323.33,192) and (322.81,191.53) .. (322.81,190.96) -- cycle ;
\draw    (324.2,213.2) -- (323.99,196) ;
\draw [shift={(323.97,194)}, rotate = 89.31] [color={rgb, 255:red, 0; green, 0; blue, 0 }  ][line width=0.75]    (4.37,-1.32) .. controls (2.78,-0.56) and (1.32,-0.12) .. (0,0) .. controls (1.32,0.12) and (2.78,0.56) .. (4.37,1.32)   ;
\draw    (322.67,100.5) -- (322.61,86) ;
\draw [shift={(322.6,84)}, rotate = 89.77] [color={rgb, 255:red, 0; green, 0; blue, 0 }  ][line width=0.75]    (4.37,-1.32) .. controls (2.78,-0.56) and (1.32,-0.12) .. (0,0) .. controls (1.32,0.12) and (2.78,0.56) .. (4.37,1.32)   ;
\draw   (296.25,181.86) .. controls (296.25,180.58) and (309.01,179.55) .. (324.75,179.55) .. controls (340.49,179.55) and (353.25,180.58) .. (353.25,181.86) .. controls (353.25,183.14) and (340.49,184.18) .. (324.75,184.18) .. controls (309.01,184.18) and (296.25,183.14) .. (296.25,181.86) -- cycle ;
\draw   (336.23,183.28) -- (341.81,183.75) -- (336.64,184.75) ;
\draw   (316.46,178.84) -- (311.35,179.88) -- (316.94,180.3) ;

\draw [color={rgb, 255:red, 3; green, 0; blue, 0 }  ,draw opacity=1 ]   (323.97,181.13) -- (323.97,187.92) ;
\draw [shift={(323.97,189.92)}, rotate = 270] [color={rgb, 255:red, 3; green, 0; blue, 0 }  ,draw opacity=1 ][line width=0.75]    (4.37,-1.32) .. controls (2.78,-0.56) and (1.32,-0.12) .. (0,0) .. controls (1.32,0.12) and (2.78,0.56) .. (4.37,1.32)   ;
\draw   (295.58,113.53) .. controls (295.58,112.25) and (308.34,111.21) .. (324.08,111.21) .. controls (339.82,111.21) and (352.58,112.25) .. (352.58,113.53) .. controls (352.58,114.81) and (339.82,115.85) .. (324.08,115.85) .. controls (308.34,115.85) and (295.58,114.81) .. (295.58,113.53) -- cycle ;
\draw   (335.57,114.95) -- (341.14,115.42) -- (335.98,116.42) ;
\draw   (315.79,110.5) -- (310.68,111.54) -- (316.28,111.97) ;

\draw  [draw opacity=0][fill={rgb, 255:red, 255; green, 255; blue, 255 }  ,fill opacity=1 ] (324.77,145.65) .. controls (324.77,144.32) and (325.85,143.23) .. (327.19,143.23) .. controls (328.53,143.23) and (329.62,144.32) .. (329.62,145.65) .. controls (329.62,146.99) and (328.53,148.08) .. (327.19,148.08) .. controls (325.85,148.08) and (324.77,146.99) .. (324.77,145.65) -- cycle ;
\draw    (322.57,108.22) .. controls (321.95,117.63) and (323.1,134.09) .. (326.6,146.1) .. controls (330.1,158.11) and (337.6,166.1) .. (345,170.3) .. controls (351.81,174.16) and (348.64,177.28) .. (341.63,177.62) ;
\draw [shift={(339.72,177.64)}, rotate = 1.17] [color={rgb, 255:red, 0; green, 0; blue, 0 }  ][line width=0.75]    (4.37,-1.32) .. controls (2.78,-0.56) and (1.32,-0.12) .. (0,0) .. controls (1.32,0.12) and (2.78,0.56) .. (4.37,1.32)   ;

\end{tikzpicture}